\documentclass[11pt,reqno]{amsart}
\usepackage{amsmath,amsthm,amsfonts,amssymb,mathrsfs}
\date{\today}

\usepackage{color}

\usepackage{hyperref}

\newtheorem{theorem}{Theorem}[section]
\newtheorem{proposition}[theorem]{Proposition}
\newtheorem{corollary}[theorem]{Corollary}
\newtheorem{lemma}[theorem]{Lemma}
\theoremstyle{definition}

\newtheorem{remark}[theorem]{Remark}

\newtheorem{definition}[theorem]{Definition}

\begin{document}

\title[Congruences on bicyclic extensions of a linearly ordered
group~]{Congruences on bicyclic extensions of a linearly ordered
group}

\author[O. Gutik]{Oleg~Gutik}
\address{Department of Mechanics and Mathematics, Ivan Franko Lviv
National University, Universytetska 1, Lviv, 79000, Ukraine}
\email{o\_\,gutik@franko.lviv.ua, ovgutik@yahoo.com}

\author[D. Pagon]{Du\v{s}an~Pagon}
\address{Faculty of Natural Sciences and Mathematics, University of
Maribor, Gos\-posvetska 84, Maribor 2000, Slovenia}
\email{dusan.pagon@uni-mb.si}

\author[K. Pavlyk]{Kateryna~Pavlyk}
\address{Institute of Mathematics,
University of Tartu, J. Liivi 2, 50409, Tartu, Estonia}
\email{kateryna.pavlyk@ut.ee}



\begin{abstract}
In the paper we study inverse semigroups $\mathscr{B}(G)$,
$\mathscr{B}^+(G)$, $\overline{\mathscr{B}}(G)$ and
$\overline{\mathscr{B}}\,^+(G)$ which are generated by partial
monotone injective translations of a positive cone of a linearly
ordered group $G$. We describe Green's relations on the semigroups
$\mathscr{B}(G)$, $\mathscr{B}^+(G)$, $\overline{\mathscr{B}}(G)$
and $\overline{\mathscr{B}}\,^+(G)$, their bands and show that they
are simple, and moreover the semigroups $\mathscr{B}(G)$ and
$\mathscr{B}^+(G)$ are bisimple. We show that for a commutative
linearly ordered group $G$ all non-trivial congruences on the
semigroup $\mathscr{B}(G)$ (and $\mathscr{B}^+(G)$) are group
congruences if and only if the group $G$ is archimedean. Also we
describe the structure of group congruences on the semigroups
$\mathscr{B}(G)$, $\mathscr{B}^+(G)$, $\overline{\mathscr{B}}(G)$
and $\overline{\mathscr{B}}\,^+(G)$.
\end{abstract}

\maketitle


\section{Introduction and main definitions}

In this article we shall follow the terminology of
\cite{Birkhoff1973, CliffordPreston1961-1967, Howie1995,
KokorinKopytov1972, Petrich1984}.

A \emph{semigroup} is a non-empty set with a binary associative
operation. A semigroup $S$ is called \emph{inverse} if for any $x\in
S$ there exists a unique $y\in S$ such that $x\cdot y\cdot x=x$ and
$y\cdot x\cdot y=y$. Such an element $y$ in $S$ is called the
\emph{inverse} of $x$ and denoted by $x^{-1}$. The map defined on an
inverse semigroup $S$ which maps every element $x$ of $S$ to its
inverse $x^{-1}$ is called the \emph{inversion}.

If $S$ is a semigroup, then we shall denote the subset of
idempotents in $S$ by $E(S)$. If $S$ is an inverse semigroup, then
$E(S)$ is closed under multiplication and we shall refer to $E(S)$
as the \emph{band of} $S$. If the band $E(S)$ is a non-empty subset
of $S$, then the semigroup operation on $S$ determines the following
partial order $\preccurlyeq$ on $E(S)$: $e\preccurlyeq f$ if and
only if $ef=fe=e$. This order is called the {\em natural partial
order} on $E(S)$. A \emph{semilattice} is a commutative semigroup of
idempotents. A semilattice $E$ is called {\em linearly ordered} or a
\emph{chain} if its natural order is a linear order.

If $\mathfrak{C}$ is an arbitrary congruence on a semigroup $S$,
then we denote by $\Phi_\mathfrak{C}\colon S\rightarrow
S/\mathfrak{C}$ the natural homomorphisms from $S$ onto the quotient
semigroup $S/\mathfrak{C}$. Also we denote by $\Omega_S$ and
$\Delta_S$ the \emph{universal} and the \emph{identity} congruences,
respectively, on the semigroup $S$, i.e., $\Omega(S)=S\times S$ and
$\Delta(S)=\{(s,s)\mid s\in S\}$. A congruence $\mathfrak{C}$ on a
semigroup $S$ is called \emph{non-trivial} if $\mathfrak{C}$ is
distinct from the universal and the identity congruence on $S$, and
a \emph{group congruence} if the quotient semigroup $S/\mathfrak{C}$
is a group. Every inverse semigroup $S$ admits a least  group
congruence $\mathfrak{C}_{mg}$:
\begin{equation*}
    a\mathfrak{C}_{mg}b \; \hbox{ if and only if there exists }\;
    e\in E(S) \; \hbox{ such that }\; ae=be
\end{equation*}
(see \cite[Lemma~III.5.2]{Petrich1984}).

A map $h\colon S\rightarrow T$ from a semigroup $S$ to a semigroup
$T$ is said to be an \emph{antihomomorphism} if $(a\cdot
b)h=(b)h\cdot(a)h$. A bijective antihomomorphism is called an
\emph{antiisomorphism}.

If $S$ is a semigroup, then we shall denote by $\mathscr{R}$,
$\mathscr{L}$, $\mathscr{J}$, $\mathscr{D}$ and $\mathscr{H}$ the
Green's relations on $S$ (see \cite{CliffordPreston1961-1967}):
\begin{align*}
    &\qquad a\mathscr{R}b \mbox{ if and only if } aS^1=bS^1;\\
    &\qquad a\mathscr{L}b \mbox{ if and only if } S^1a=S^1b;\\
    &\qquad a\mathscr{J}b \mbox{ if and only if } S^1aS^1=S^1bS^1;\\
    &\qquad \mathscr{D}=\mathscr{L}\circ\mathscr{R}=\mathscr{R}\circ
    \mathscr{L};\\
    &\qquad \mathscr{H}=\mathscr{L}\cap\mathscr{R}.
\end{align*}

Let $\mathscr{I}_X$ denote the set of all partial one-to-one
transformations of an infinite set $X$  together with the following
semigroup operation: $x(\alpha\beta)=(x\alpha)\beta$ if
$x\in\operatorname{dom}(\alpha\beta)=\{
y\in\operatorname{dom}\alpha\mid
y\alpha\in\operatorname{dom}\beta\}$,  for
$\alpha,\beta\in\mathscr{I}_X$. The semigroup $\mathscr{I}_X$ is
called the \emph{symmetric inverse semigroup} over the set $X$~(see
\cite{CliffordPreston1961-1967}). The symmetric inverse semigroup
was introduced by Wagner~\cite{Wagner1952} and it plays a major role
in the theory of semigroups.

The bicyclic semigroup ${\mathscr{C}}(p,q)$ is the semigroup with
the identity $1$ generated by two elements $p$ and $q$ subjected
only to the condition $pq=1$. The distinct elements of
${\mathscr{C}}(p,q)$ are exhibited in the following useful array
\begin{equation*}
\begin{array}{ccccc}
  1      & p      & p^2    & p^3    & \cdots \\
  q      & qp     & qp^2   & qp^3   & \cdots \\
  q^2    & q^2p   & q^2p^2 & q^2p^3 & \cdots \\
  q^3    & q^3p   & q^3p^2 & q^3p^3 & \cdots \\
  \vdots & \vdots & \vdots & \vdots & \ddots \\
\end{array}
\end{equation*}
and the semigroup operation on ${\mathscr{C}}(p,q)$ is determined as
follows:
\begin{equation*}
    q^kp^l\cdot q^mp^n=q^{k+m-\min\{l,m\}}p^{l+n-\min\{l,m\}}.
\end{equation*}
The bicyclic semigroup plays an important role in the algebraic
theory of semigroups and in the theory of topological semigroups.
For example the well-known O.~Andersen's result~\cite{Andersen1952}
states that a ($0$--) simple semigroup is completely ($0$--) simple
if and only if it does not contain the bicyclic semigroup. The
bicyclic semigroup does not embed into stable
semigroups~\cite{KochWallace1957}.

\begin{remark}\label{remark-1.0}
We observe that the bicyclic semigroup is isomorphic to the
semigroup $\mathscr{C}_{\mathbb{N}}(\alpha,\beta)$ which is
generated by injective partial transformations $\alpha$ and $\beta$
of the set of positive integers $\mathbb{N}$, defined as follows:
\begin{equation*}
 \begin{split}
    (n)\alpha=n+1 & \quad \mbox{ if } \, n\geqslant 1;\\
    (n)\beta=n-1  & \quad \mbox{ if } \, n> 1
\end{split}
\end{equation*}
(see Exercise~IV.1.11$(ii)$ in \cite{Petrich1984}).
\end{remark}

Recall from~\cite{Fuchs1963} that a \emph{partially-ordered group}
is a group $(G,\cdot)$ equipped with a partial order  $\leqslant$
that is translation-invariant; in other words,  $\leqslant$  has the
property that, for all $a, b, g\in G$, if $a\leqslant b$ then
$a\cdot g\leqslant b\cdot g$ and $g\cdot a\leqslant g\cdot b$.

Later by $e$ we denote the identity of a group $G$. The set
$G^+=\{x\in G\mid e\leqslant x\}$ in a partially ordered group $G$
is called the \emph{positive cone}, or the \emph{integral part}, of
$G$ and satisfies the properties:
\begin{itemize}
  \item[1)] $G^+\cdot G^+\subseteq G^+$;
  \item[2)] $G^+\cap (G^+)^{-1}=\{e\}$; \; and
  \item[3)] $x^{-1}\cdot G^+\cdot x\subseteq G^+$ for all $x\in G$.
\end{itemize}
Any subset $P$ of a group $G$ that satisfies the conditions 1)--3)
induces a partial order on $G$ ($x\leqslant y$ if and only if
$x^{-1}\cdot y\in P$) for which $P$ is the positive cone.

A \emph{linearly ordered} or \emph{totally ordered group} is an
ordered group $G$ such that the order relation ``$\leqslant$'' is
total~\cite{Birkhoff1973}.

In the remainder we shall assume that $G$ is a linearly ordered
group.

For every $g\in G$ we denote
\begin{equation*}
    G^+(g)=\{x\in G\mid g\leqslant x\}.
\end{equation*}
The set $G^+(g)$ is called a \emph{positive cone on element} $g$ in
$G$.

For arbitrary elements $g,h\in G$ we consider a partial map
$\alpha_h^g\colon G\rightharpoonup G$ defined by the formula
\begin{equation*}
    (x)\alpha_h^g=x\cdot g^{-1}\cdot h, \qquad \hbox{ for } \; x\in
    G^{+}(g).
\end{equation*}
We observe that Lemma~XIII.1 from \cite{Birkhoff1973} implies that
for such partial map $\alpha_h^g\colon G\rightharpoonup G$ the
restriction $\alpha_h^g\colon G^+(g)\rightarrow G^+(h)$ is a
bijective map.

We denote
\begin{equation*}
    \mathscr{B}(G)=\{\alpha_h^g\colon G\rightharpoonup G\mid g,h\in
    G\} \, \hbox{ and } \,
    \mathscr{B}^+(G)=\{\alpha_h^g\colon G\rightharpoonup G\mid g,h\in
    G^+\},
\end{equation*}
and consider on the sets $\mathscr{B}(G)$ and $\mathscr{B}^+(G)$ the
operation of the composition of partial maps. Simple verifications
show that
\begin{equation}\label{formula-1.1}
\alpha_h^g\cdot \alpha^k_l=\alpha^a_b, \qquad \hbox{ where } \quad
a=(h\vee k)\cdot h^{-1}\cdot g \quad \hbox{ and } \quad b=(h\vee
k)\cdot k^{-1}\cdot l,
\end{equation}
for $g,h,k,l\in G$. Therefore, property 1) of the positive cone and
condition~(\ref{formula-1.1}) imply that $\mathscr{B}(G)$ and
$\mathscr{B}^+(G)$ are subsemigroups of $\mathscr{I}_G$.

\begin{proposition}\label{proposition-1.1}
Let $G$ be a linearly ordered group. Then the following assertions
hold:
\begin{itemize}
  \item[$(i)$] elements $\alpha_h^g$ and $\alpha_g^h$ are inverses
   of each other in $\mathscr{B}(G)$ for all $g,h\in G$
   $($resp., $\mathscr{B}^+(G)$ for all $g,h\in G^+)$;

  \item[$(ii)$] an element $\alpha_h^g$ of the semigroup
   $\mathscr{B}(G)$ $($resp., $\mathscr{B}^+(G))$ is an
   idempotent if and only if $g=h$;

  \item[$(iii)$] $\mathscr{B}(G)$ and $\mathscr{B}^+(G)$ are inverse
   subsemigroups of $\mathscr{I}_G$;

  \item[$(iv)$] the semigroup $\mathscr{B}(G)$ $($resp.,
   $\mathscr{B}^+(G))$ is isomorphic to $S_G=G\times G$ $($resp.,
   $S_G^+=G^+\times G^+)$ with the semigroup operation:
  \begin{equation*}
  (a,b)\cdot(c,d)=
  \left\{
    \begin{array}{ll}
      (c\cdot b^{-1}\cdot a,d), & \hbox{if }  \; b<c;\\
      (a,d), & \hbox{if } \; b=c; \\
      (a,b\cdot c^{-1}\cdot d), & \hbox{if } \; b>c,
    \end{array}
  \right.
  \end{equation*}
  where $a,b,c,d\in G$ $($resp.,  $a,b,c,d\in G^+$$)$.
\end{itemize}
\end{proposition}

\begin{proof}
$(i)$ Condition ~(\ref{formula-1.1}) implies that
\begin{equation*}
\alpha_h^g\cdot\alpha_g^h\cdot\alpha_h^g=\alpha_h^g \qquad
\hbox{and} \qquad
\alpha_g^h\cdot\alpha_h^g\cdot\alpha_g^h=\alpha_g^h,
\end{equation*}
and hence $\alpha_h^g$ and $\alpha_g^h$ are inverse elements in
$\mathscr{B}(G)$ (resp., $\mathscr{B}^+(G)$).

Statement $(ii)$ follows from the property of the semigroup
$\mathscr{I}_G$ that $\alpha\in\mathscr{I}_G$ is an idempotent if
and only if $\alpha\colon\operatorname{dom}\alpha\rightarrow
\operatorname{ran}\alpha$ is an identity map.

Statements $(i)$, $(ii)$ and Theorem~1.17 from
\cite{CliffordPreston1961-1967} imply statement $(iii)$.

Statement $(iv)$ is a corollary of condition~(\ref{formula-1.1}).
\end{proof}

\begin{remark}\label{remark-1.2}
We observe that Proposition~\ref{proposition-1.1} implies that:
\begin{itemize}
  \item[$(1)$] if $G$ is the additive group of integers
  $(\mathbb{Z},+)$ with usual linear order $\leqslant$ then the
  semigroup $\mathscr{B}^+(G)$ is isomorphic to the bicyclic
  semigroup ${\mathscr{C}}(p,q)$;

  \item[$(2)$] if $G$ is the additive group of real numbers
  $(\mathbb{R},+)$ with usual linear order $\leqslant$ then the
  semigroup $\mathscr{B}(G)$ is isomorphic to
  $B^{2}_{(-\infty,\infty)}$ (see \cite{Korkmaz1997, Korkmaz2009})
  and the semigroup $\mathscr{B}^+(G)$ is isomorphic to
  $B^{1}_{[0,\infty)}$ (see
  \cite{Ahre1981, Ahre1983, Ahre1984, Ahre1986, Ahre1989}) \; and

  \item[$(3)$] the semigroup $\mathscr{B}^+(G)$ is isomorphic to
  the semigroup $S(G)$ which is defined in \cite{Fotedar1974,
  Fotedar1978}.
\end{itemize}
\end{remark}

We shall say that a linearly ordered group $G$ is a \emph{$d$-group}
if for every element $g\in G^+\setminus\{e\}$ there exists $x\in
G^+\setminus\{e\}$ such that $x<g$. We observe that a linearly
ordered group $G$ is a $d$-group if and only if the set
$G^+\setminus\{e\}$ does not contain a minimal element.

\begin{definition}\label{definition-1.3}
Suppose that $G$ is a linearly ordered $d$-group. For every $g\in G$
we denote
\begin{equation*}
    \overset{_\circ}{G}{}^+(g)=\{x\in G\mid g<x\}.
\end{equation*}
The set $\overset{_\circ}{G}{}^+(g)$ is called a
\emph{$\circ$-positive cone on element} $g$ in $G$.

For arbitrary elements $g,h\in G$ we consider a partial map
$\overset{_\circ}{\alpha}{}^{g}_h\colon G\rightharpoonup G$ defined
by the formula
\begin{equation*}
    (x)\overset{_\circ}{\alpha}{}^{g}_h=x\cdot g^{-1}\cdot h,
    \qquad \hbox{ for } \; x\in\overset{_\circ}{G}{}^+(g).
\end{equation*}
We observe that Lemma~XIII.1 from \cite{Birkhoff1973} implies that
for such partial map $\overset{_\circ}{\alpha}{}^{g}_h\colon
G\rightharpoonup G$ the restriction
$\overset{_\circ}{\alpha}{}^{g}_h \colon \overset{_\circ}{G}{}^+(g)
\rightarrow\overset{_\circ}{G}{}^+(h)$ is a bijective map.

We denote
\begin{equation*}
    \overset{_\circ}{\mathscr{B}}(G)=
    \left\{\overset{_\circ}{\alpha}{}^g_h\colon G\rightharpoonup G
    \mid g,h\in
    G\right\} \, \hbox{ and } \,
    \overset{_\circ}{\mathscr{B}}{}^+(G)=
    \left\{\overset{_\circ}{\alpha}{}^g_h\colon G\rightharpoonup G
    \mid g,h\in
    G^+\right\},
\end{equation*}
and consider on the sets $\overset{_\circ}{\mathscr{B}}(G)$ and
$\overset{_\circ}{\mathscr{B}}{}^+(G)$ the operation of the
composition of partial maps. Simple verifications show that
\begin{equation}\label{formula-1.2}
\overset{_\circ}{\alpha}{}^g_h\cdot
\overset{_\circ}{\alpha}{}^k_l=\overset{_\circ}{\alpha}{}^a_b,
\qquad \hbox{ where } \quad a=(h\vee k)\cdot h^{-1}\cdot g \quad
\hbox{ and } \quad b=(h\vee k)\cdot k^{-1}\cdot l,
\end{equation}
for $g,h,k,l\in G$. Therefore, property 1) of the positive cone and
condition~(\ref{formula-1.2}) imply that
$\overset{_\circ}{\mathscr{B}}(G)$ and
$\overset{_\circ}{\mathscr{B}}{}^+(G)$ are subsemigroups of the
symmetric inverse semigroup $\mathscr{I}_G$.
\end{definition}

\begin{proposition}\label{proposition-1.4}
If $G$ is a linearly ordered $d$-group then the semigroups
$\overset{_\circ}{\mathscr{B}}(G)$ and
$\overset{_\circ}{\mathscr{B}}{}^+(G)$ are isomorphic to
$\mathscr{B}(G)$ and $\mathscr{B}^+(G)$, respectively.
\end{proposition}

\begin{proof}
A map $\mathfrak{h}\colon\mathscr{B}(G)
\rightarrow\overset{_\circ}{\mathscr{B}}(G)$ (resp.,
$\mathfrak{h}\colon \mathscr{B}^+(G)
\rightarrow\overset{_\circ}{\mathscr{B}}{}^+(G)$) we define by the
formula:
\begin{equation*}
    (\alpha^g_h)\mathfrak{h}=\overset{_\circ}{\alpha}{}^g_h, \qquad
    \hbox{ for } \; g,h\in G \; (\hbox{resp.,~}g,h\in G^+).
\end{equation*}
Simple verifications show that such map $\mathfrak{h}$ is an
isomorphism of the semigroups $\overset{_\circ}{\mathscr{B}}(G)$ and
$\mathscr{B}(G)$ (resp., $\overset{_\circ}{\mathscr{B}}{}^+(G)$ and
$\mathscr{B}^+(G)$).
\end{proof}

Suppose that $G$ is a linearly ordered $d$-group. Then obviously
$\overset{_\circ}{\mathscr{B}}(G)\cap{\mathscr{B}}(G)= \varnothing$
and $\overset{_\circ}{\mathscr{B}}{}^+(G)\cap
\mathscr{B}^+(G)=\varnothing$. We define
\begin{equation*}
    \overline{\mathscr{B}}(G)=\overset{_\circ}{\mathscr{B}}(G)
    \cup{\mathscr{B}}(G) \qquad \hbox{ and } \qquad
    \overline{\mathscr{B}}\,^+(G)=\overset{_\circ}{\mathscr{B}}{}^+(G)
    \cup \mathscr{B}^+(G).
\end{equation*}

\begin{proposition}\label{proposition-1.5}
If $G$ is a linearly ordered $d$-group then
$\overline{\mathscr{B}}(G)$ and $\overline{\mathscr{B}}\,^+(G)$ are
inverse semigroups.
\end{proposition}

\begin{proof}
Since $\overset{_\circ}{\mathscr{B}}(G)$, $\mathscr{B}(G)$,
$\overset{_\circ}{\mathscr{B}}{}^+(G)$ and $\mathscr{B}^+(G)$ are
inverse subsemigroups of the symmetric inverse semigroup
$\mathscr{I}_G$ over the group $G$ we conclude that it is sufficient
to show that $\overline{\mathscr{B}}(G)$ and
$\overline{\mathscr{B}}\,^+(G)$ are subsemigroups of
$\mathscr{I}_G$.

We fix arbitrary elements $g,h,k,l\in G$. Since $\alpha^g_h$,
$\alpha^k_l$, $\overset{_\circ}{\alpha}{}^g_h$ and
$\overset{_\circ}{\alpha}{}^k_l$ are partial injective maps from $G$
into $G$ we have that
\begin{equation*}
    \alpha^g_h\cdot\overset{_\circ}{\alpha}{}^k_l=
\left\{
  \begin{array}{rl}
    \overset{_\circ}{\alpha}{}^{k\cdot h^{-1}\cdot g}_l,
                & \hbox{ if }\; h<k; \\
    \overset{_\circ}{\alpha}{}^g_l,
                & \hbox{ if }\; h=k; \\
    \alpha^g_{h\cdot k^{-1}\cdot l},
                & \hbox{ if }\; h>k
  \end{array}
\right. \, \hbox{ and~ } \,
    \overset{_\circ}{\alpha}{}^g_h\cdot\alpha^k_l=
\left\{
  \begin{array}{rl}
    \alpha^{k\cdot h^{-1}\cdot g}_l,
                & \hbox{ if }\; h<k; \\
    \overset{_\circ}{\alpha}{}^g_l,
                & \hbox{ if }\; h=k; \\
    \overset{_\circ}{\alpha}{}^g_{h\cdot k^{-1}\cdot l},
                & \hbox{ if }\; h>k.
  \end{array}
\right.
\end{equation*}
Hence $\overline{\mathscr{B}}(G)$ is a subsemigroup of
$\mathscr{I}_G$.

Similar arguments and property 1) of the positive cone imply that
$\overline{\mathscr{B}}\,^+(G)$ is a subsemigroup of
$\mathscr{I}_G$. This completes the proof of our proposition.
\end{proof}

In our paper we study semigroups $\mathscr{B}(G)$ and
$\mathscr{B}^+(G)$ for a linearly ordered group $G$, and semigroups
$\overline{\mathscr{B}}(G)$ and $\overline{\mathscr{B}}\,^+(G)$ for
a linearly ordered $d$-group $G$. We describe Green's relations on
the semigroups $\mathscr{B}(G)$, $\mathscr{B}^+(G)$,
$\overline{\mathscr{B}}(G)$ and $\overline{\mathscr{B}}\,^+(G)$,
their bands and show that they are simple, and moreover the
semigroups $\mathscr{B}(G)$ and $\mathscr{B}^+(G)$ are bisimple. We
show that for a commutative linearly ordered group $G$ all
non-trivial congruences on the semigroup $\mathscr{B}(G)$ (and
$\mathscr{B}^+(G)$) are group congruences if and only if the group
$G$ is archimedean. Also, we describe the structure of group
congruences on the semigroups $\mathscr{B}(G)$, $\mathscr{B}^+(G)$,
$\overline{\mathscr{B}}(G)$ and $\overline{\mathscr{B}}\,^+(G)$.

\section{Algebraic properties of the semigroups $\mathscr{B}(G)$ and
$\mathscr{B}^+(G)$}

\begin{proposition}\label{proposition-2.1} Let $G$ be a linearly
ordered group. Then the following assertions hold:
\begin{itemize}
  \item[$(i)$] if $\alpha^g_g,\alpha^h_h\in E(\mathscr{B}(G))$
   $\big($resp., $\alpha^g_g,\alpha^h_h\in
   E(\mathscr{B}^+(G))$$\big)$,
       then $\alpha^g_g\preccurlyeq\alpha^h_h$ if and only if
       $g\geqslant h$ in $G$ $\big($resp., in $G^+$$\big)$;

  \item[$(ii)$] the semilattice $E(\mathscr{B}(G))$
   $\big($resp., $E(\mathscr{B}^+(G))$$\big)$ is isomorphic to
       $G$  $\big($resp., $G^+$$\big)$, considered as a $\vee$-semilattice
       under the mapping
       $(\alpha^g_g)\mathfrak{i}=g$;

  \item[$(iii)$] $\alpha^g_h\mathscr{R}\alpha^k_l$ in
   $\mathscr{B}(G)$ $\big($resp., in $\mathscr{B}^+(G)$$\big)$
   if and only if $g=k$ in $G$ $\big($resp., in $G^+$$\big)$;

  \item[$(iv)$] $\alpha^g_h\mathscr{L}\alpha^k_l$ in
       $\mathscr{B}(G)$ $\big($resp., in $\mathscr{B}^+(G)$$\big)$
       if and only if $h=l$ in $G$ $\big($resp., in $G^+$$\big)$;

  \item[$(v)$] $\alpha^g_h\mathscr{H}\alpha^k_l$ in
       $\mathscr{B}(G)$ $\big($resp., in $\mathscr{B}^+(G)$$\big)$
       if and only
       if $g=k$ and $h=l$ in $G$ $\big($resp., in $G^+$$\big)$, and
       hence every $\mathscr{H}$-class in $\mathscr{B}(G)$
       $\big($resp., in $\mathscr{B}^+(G)$$\big)$ is a singleton set;

  \item[$(vi)$] $\alpha^g_h\mathscr{D}\alpha^k_l$ in
       $\mathscr{B}(G)$ $\big($resp., in $\mathscr{B}^+(G)$$\big)$
       for all $g,h,k,l\in G$ and
       hence $\mathscr{B}(G)$ $\big($resp., $\mathscr{B}^+(G)$$\big)$
       is a bisimple semigroup;

  \item[$(vii)$] $\mathscr{B}(G)$ $\big($resp.,
   $\mathscr{B}^+(G)$$\big)$ is a simple
   semigroup.
\end{itemize}
\end{proposition}

\begin{proof}
Statements $(i)$ and $(ii)$ are trivial and follow from the
definition of the semigroup $\mathscr{B}(G)$.

$(iii)$ Let $\alpha^g_h,\alpha^k_l\in\mathscr{B}(G)$ be such that
$\alpha^g_h\mathscr{R}\alpha^k_l$. Since $\alpha^g_h\mathscr{B}(G)=
\alpha^k_l\mathscr{B}(G)$ and $\mathscr{B}(G)$ is an inverse
semigroup, Theorem~1.17 from \cite{CliffordPreston1961-1967} implies
that $\alpha^g_h\mathscr{B}(G)=
\alpha^g_h(\alpha^g_h)^{-1}\mathscr{B}(G)$ and
$\alpha^k_l\mathscr{B}(G)=
\alpha^k_l(\alpha^k_l)^{-1}\mathscr{B}(G)$, and hence
$\alpha^g_g=\alpha^g_h(\alpha^g_h)^{-1}=
\alpha^k_l(\alpha^k_l)^{-1}=\alpha^k_k$. Therefore we get that
$g=k$.

Conversely, let $\alpha^g_h,\alpha^k_l\in\mathscr{B}(G)$ be such
that $g=k$. Then $\alpha^g_h(\alpha^g_h)^{-1}=
\alpha^k_l(\alpha^k_l)^{-1}$. Since $\mathscr{B}(G)$ is an inverse
semigroup, Theorem~1.17 from \cite{CliffordPreston1961-1967} implies
that $\alpha^g_h\mathscr{B}(G)=
\alpha^g_h(\alpha^g_h)^{-1}\mathscr{B}(G)= \alpha^k_l\mathscr{B}(G)$
and hence $\alpha^g_h\mathscr{R}\alpha^k_l$ in $\mathscr{B}(G)$.

The proof of statement $(iv)$ is similar to $(iii)$.

Statement $(v)$ follows from statements $(iii)$ and $(iv)$.

$(vi)$ For every $g,h\in\mathscr{B}(G)$ we have that
$\alpha^g_h(\alpha^g_h)^{-1}=\alpha^g_g$ and
$(\alpha^g_h)^{-1}\alpha^g_h=\alpha^h_h$, and hence by statement
$(ii)$, Proposition~\ref{proposition-1.1} and Lemma~1.1 from
\cite{Munn1966} we get that $\mathscr{B}(G)$ is a bisimple
semigroup.

$(vii)$ Since every two $\mathscr{D}$-equivalent elements of an
arbitrary semigroup $S$ are $\mathscr{J}$-equivalent (see
\cite[Section~2.1]{CliffordPreston1961-1967}) we have that
$\mathscr{B}(G)$ is a simple semigroup.

The proof of the proposition for the semigroup $\mathscr{B}^+(G)$ is
similar.
\end{proof}

Given two partially ordered sets $(A,\leqslant_A)$ and
$(B,\leqslant_B)$, the \emph{lexicographical order}
$\leqslant_\textbf{lex}$ on the Cartesian product $A\times B$ is
defined as follows:
\begin{equation*}
   (a,b)\leqslant_\textbf{lex} (a^{\prime},b^{\prime}) \qquad
   \hbox{ if and only if } \qquad a<_A a^{\prime} \quad
   \hbox{ or } \quad (a = a^{\prime} \hbox{ and }
   b\leqslant_B b^{\prime}).
\end{equation*}
In this case we shall say that the partially ordered set $(A\times
B,\leqslant_\textbf{lex})$ is the \emph{lexicographic product} of
partially ordered sets $(A,\leqslant_A)$ and $(B,\leqslant_B)$ and
it is denoted by $A\times_\textbf{lex} B$. We observe that the
lexicographic product of two linearly ordered sets is a linearly
ordered set.

\begin{proposition}\label{proposition-2.8} Let $G$ be a linearly
ordered $d$-group. Then the following assertions hold:
\begin{itemize}
  \item[$(i)$] $E\left(\overline{\mathscr{B}}(G)\right)=
       E(\mathscr{B}(G))\cup
       E\big(\overset{_\circ}{\mathscr{B}}(G)\big)$ and
       $E\left(\overline{\mathscr{B}}\,^+(G)\right)=
       E(\mathscr{B}\,^+(G))\cup
       E\big(\overset{_\circ}{\mathscr{B}}\,^+(G)\big)$.

  \item[$(ii)$] If $\alpha^g_g,\overset{_\circ}{\alpha}{}^g_g,
       \alpha^h_h, \overset{_\circ}{\alpha}{}^h_h\in
       E\left(\overline{\mathscr{B}}(G)\right)$ $\big($resp.,
       $\alpha^g_g,\overset{_\circ}{\alpha}{}^g_g,
       \alpha^h_h, \overset{_\circ}{\alpha}{}^h_h\in
       E\left(\overline{\mathscr{B}}\,^+(G)\right)$$\big)$ then:
    \begin{itemize}
      \item[(a)] $\alpha^g_g\preccurlyeq\alpha^h_h$ if and only if
           $g\geqslant h$ in $G$ $\big($resp., in $G^+$$\big)$;
      \item[(b)] $\overset{_\circ}{\alpha}{}^g_g\preccurlyeq
           \overset{_\circ}{\alpha}{}^h_h$ if and only if
           $g\geqslant h$ in $G$ $\big($resp., in $G^+$$\big)$;
      \item[(c)] $\alpha^g_g\preccurlyeq
           \overset{_\circ}{\alpha}{}^h_h$ if and only if $g>h$ in
           $G$ $\big($resp., in $G^+$$\big)$;
      \item[(d)] $\overset{_\circ}{\alpha}{}^g_g\preccurlyeq
           \alpha^h_h$ if and only if $g\geqslant h$ in $G$
           $\big($resp., in $G^+$$\big)$.
    \end{itemize}

  \item[$(iii)$] The semilattice
       $E\left(\overline{\mathscr{B}}(G)\right)$
       $\big($resp., $E\left(\overline{\mathscr{B}}\,^+(G)\right)$$\big)$
       is isomorphic to the lexicographic product
       $G\times_{\operatorname{\textbf{lex}}}
       \{0,1\}$ $\big($resp., $G^+\times_{\operatorname{\textbf{lex}}}
       \{0,1\}$$\big)$ of semilattices $(G,\vee)$ $\big($resp.,
       $(G^+,\vee)$$\big)$ and $(\{0,1\},\min)$ under
       the mapping $(\alpha^g_g)\mathfrak{i}=(g,1)$ and
       $(\overset{_\circ}{\alpha}{}^g_g)\mathfrak{i}=(g,0)$,
       and hence $E\left(\overline{\mathscr{B}}(G)\right)$
       $\big($resp.,\break
       $E\left(\overline{\mathscr{B}}\,^+(G)\right)$$\big)$ is a
       linearly ordered semilattice.

  \item[$(iv)$] The elements $\alpha$ and $\beta$ of the semigroup
      $\overline{\mathscr{B}}(G)$ $\big($resp.,
      $\overline{\mathscr{B}}\,^+(G)$$\big)$ are
      $\mathscr{R}$-equivalent in
      $\overline{\mathscr{B}}(G)$ $\big($resp., in
      $\overline{\mathscr{B}}\,^+(G)$$\big)$ provides either
      $\alpha,\beta\in\mathscr{B}(G)$ $\big($resp.,
      $\alpha,\beta\in\overline{\mathscr{B}}\,^+(G)$$\big)$ or
      $\alpha,\beta\in\overset{_\circ}{\mathscr{B}}(G)$
      $\big($resp.,
      $\alpha,\beta\in\overset{_\circ}{\mathscr{B}}\,^+(G)$$\big)$; and
      moreover, we have that
      \begin{itemize}
        \item[(a)] $\alpha^g_h\mathscr{R}\alpha^k_l$ in
            $\overline{\mathscr{B}}(G)$ $\big($resp., in
            $\overline{\mathscr{B}}\,^+(G)$$\big)$ if and only if $g=k$;
            \; and
        \item[(b)] $\overset{_\circ}{\alpha}{}^g_h\mathscr{R}
            \overset{_\circ}{\alpha}{}^k_l$ in
            $\overline{\mathscr{B}}(G)$ $\big($resp., in
            $\overline{\mathscr{B}}\,^+(G)$$\big)$ if and only if $g=k$.
      \end{itemize}

  \item[$(v)$] The elements $\alpha$ and $\beta$ of the semigroup
      $\overline{\mathscr{B}}(G)$ $\big($resp.,
      $\overline{\mathscr{B}}\,^+(G)$$\big)$ are
      $\mathscr{L}$-equivalent in
      $\overline{\mathscr{B}}(G)$ $\big($resp., in
      $\overline{\mathscr{B}}\,^+(G)$$\big)$ provides either
      $\alpha,\beta\in\mathscr{B}(G)$ $\big($resp.,
      $\alpha,\beta\in\overline{\mathscr{B}}\,^+(G)$$\big)$ or
      $\alpha,\beta\in\overset{_\circ}{\mathscr{B}}(G)$ $\big($resp.,
      $\alpha,\beta\in\overset{_\circ}{\mathscr{B}}\,^+(G)$$\big)$;
      and moreover,  we have that
      \begin{itemize}
        \item[(a)] $\alpha^g_h\mathscr{L}\alpha^k_l$ in
            $\overline{\mathscr{B}}(G)$ $\big($resp., in
      $\overline{\mathscr{B}}\,^+(G)$$\big)$ if and only if $h=l$; \; and
        \item[(b)] $\overset{_\circ}{\alpha}{}^g_h\mathscr{L}
            \overset{_\circ}{\alpha}{}^k_l$ in
            $\overline{\mathscr{B}}(G)$ $\big($resp., in
      $\overline{\mathscr{B}}\,^+(G)$$\big)$ if and only if $h=l$.
      \end{itemize}

  \item[$(vi)$] The elements $\alpha$ and $\beta$ of the semigroup
      $\overline{\mathscr{B}}(G)$ $\big($resp.,
      $\overline{\mathscr{B}}\,^+(G)$$\big)$ are
      $\mathscr{H}$-equivalent in
      $\overline{\mathscr{B}}(G)$ $\big($resp., in
      $\overline{\mathscr{B}}\,^+(G)$$\big)$ provides either
      $\alpha,\beta\in\mathscr{B}(G)$ $\big($resp.,
      $\alpha,\beta\in\overline{\mathscr{B}}\,^+(G)$$\big)$ or
      $\alpha,\beta\in\overset{_\circ}{\mathscr{B}}(G)$ $\big($resp.,
      $\alpha,\beta\in\overset{_\circ}{\mathscr{B}}\,^+(G)$$\big)$;
      and moreover,  we have that
      \begin{itemize}
        \item[(a)] $\alpha^g_h\mathscr{H}\alpha^k_l$ in
            $\overline{\mathscr{B}}(G)$ $\big($resp., in
      $\overline{\mathscr{B}}\,^+(G)$$\big)$ if and only if $g=k$ and
            $h=l$;
        \item[(b)] $\overset{_\circ}{\alpha}{}^g_h\mathscr{H}
            \overset{_\circ}{\alpha}{}^k_l$ in
            $\overline{\mathscr{B}}(G)$ $\big($resp., in
      $\overline{\mathscr{B}}\,^+(G)$$\big)$ if and only if $g=k$ and
            $h=l$; \; and
        \item[(c)] every $\mathscr{H}$-class in
        $\overline{\mathscr{B}}(G)$
        $\big($resp., in $\overline{\mathscr{B}}\,^+(G)$$\big)$
        is a singleton  set.
      \end{itemize}

  \item[$(vii)$] $\overline{\mathscr{B}}(G)$ $\big($resp.,
    $\overline{\mathscr{B}}\,^+(G)$$\big)$ is a simple semigroup.

  \item[$(viii)$] The semigroup $\overline{\mathscr{B}}(G)$
   $\big($resp.,
   $\overline{\mathscr{B}}\,^+(G)$$\big)$ has only two distinct
   $\mathscr{D}$-classes which are inverse subsemigroups
   $\mathscr{B}(G)$ and $\overset{_\circ}{\mathscr{B}}(G)$
   $\big($resp., $\mathscr{B}^+(G)$ and
       $\overset{_\circ}{\mathscr{B}}\,^+(G)$$\big)$.
\end{itemize}
\end{proposition}

\begin{proof}
Statements $(i)$, $(ii)$ and $(iii)$ follow from the definition of
the semigroup $\overline{\mathscr{B}}(G)$ and
Proposition~\ref{proposition-1.5}.

Proofs of statements $(iv)$, $(v)$ and $(vi)$ follow from
Proposition~\ref{proposition-1.5} and
Theorem~1.17~\cite{CliffordPreston1961-1967} and are similar to
statements $(ii)$, $(iv)$ and $(v)$ of
Proposition~\ref{proposition-2.1}.

$(vii)$ We shall show that
$\overline{\mathscr{B}}(G)\cdot\alpha\cdot
\overline{\mathscr{B}}(G)=\overline{\mathscr{B}}(G)$ for every
$\alpha\in\overline{\mathscr{B}}(G)$. We fix arbitrary
$\alpha,\beta\in\overline{\mathscr{B}}(G)$ and show that there exist
$\gamma,\delta\in\overline{\mathscr{B}}(G)$ such that
$\gamma\cdot\alpha\cdot\delta=\beta$.

We consider the following cases:
\begin{itemize}
  \item[$(1)$] $\alpha=\alpha^g_h\in\mathscr{B}(G)$ and
       $\beta=\alpha^k_l\in\mathscr{B}(G)$;

  \item[$(2)$] $\alpha=\alpha^g_h\in\mathscr{B}(G)$ and
       $\beta=\overset{_\circ}{\alpha}{}^k_l\in
       \overset{_\circ}{\mathscr{B}}(G)$;

  \item[$(3)$] $\alpha=\overset{_\circ}{\alpha}{}^g_h\in
       \overset{_\circ}{\mathscr{B}}(G)$ and
       $\beta=\alpha^k_l\in\mathscr{B}(G)$;

  \item[$(4)$] $\alpha=\overset{_\circ}{\alpha}{}^g_h\in
       \overset{_\circ}{\mathscr{B}}(G)$ and
       $\beta=\overset{_\circ}{\alpha}{}^k_l\in
       \overset{_\circ}{\mathscr{B}}(G)$,
\end{itemize}
where $g,h,k,l\in G$.

We put:
\begin{itemize}
  \item[] $\gamma=\alpha^k_g$ and $\delta=\alpha^h_l$ in case $(1)$;

  \item[] $\gamma=\overset{_\circ}{\alpha}{}^k_g$ and
       $\delta=\overset{_\circ}{\alpha}{}^h_l$ in case $(2)$;

  \item[] $\gamma=\alpha^k_a$ and $\delta=\alpha^{a\cdot
       g^{-1}\cdot h}_l$, where $a\in G^+(g)\setminus\{g\}$, in case
       $(3)$;

  \item[] $\gamma=\overset{_\circ}{\alpha}{}^k_g$ and
       $\delta=\overset{_\circ}{\alpha}{}^h_l$ in case $(4)$.
\end{itemize}
Elementary verifications show that
$\gamma\cdot\alpha\cdot\delta=\beta$, and this completes the proof
of assertion $(vii)$.

Statement $(viii)$ follows from statements $(iv)$ and $(v)$.

The proof of the statements of the proposition for the semigroup
$\overline{\mathscr{B}}\,^+(G)$ is similar.
\end{proof}

\begin{proposition}\label{proposition-2.3} Let $G$ be a linearly
ordered group. Then for any distinct elements $g$ and $h$ in $G$
such that $g\leqslant h$ in $G$ $($resp., in $G^+)$ the subsemigroup
$\mathscr{C}\left(\overline{g,h}\right)$ of $\mathscr{B}(G)$
$($resp., $\mathscr{B}^+(G))$, which is generated by elements
$\alpha^g_h$ and $\alpha^h_g$, is isomorphic to the bicyclic
semigroup, and hence for every idempotent $\alpha^g_g$ in
$\mathscr{B}(G)$ $($resp., in $\mathscr{B}^+(G))$ there exists a
subsemigroup $\mathscr{C}$ in $\mathscr{B}(G)$ $($resp., in
$\mathscr{B}^+(G))$ such that $\alpha^g_g$ is a unit of
$\mathscr{C}$ and $\mathscr{C}$ is isomorphic to the bicyclic
semigroup.
\end{proposition}

\begin{proof}
Since the semigroup $\mathscr{C}$ which is generated by elements
$\alpha^g_h$ and $\alpha^h_g$ is isomorphic to the semigroup
$\mathscr{C}_{\mathbb{N}}(\alpha,\beta)$ (this isomorphism
$\mathfrak{i}\colon \mathscr{C}\rightarrow
\mathscr{C}_{\mathbb{N}}(\alpha,\beta)$ we can determine on
generating elements of $\mathscr{C}$ by the formulae
$(\alpha^g_h)\mathfrak{i}=\alpha$ and
$(\alpha^h_g)\mathfrak{i}=\beta$) we conclude that the first part of
the proposition follows from Remark~\ref{remark-1.0}. Obviously, the
element $\alpha^g_g$ is a unity of the semigroup $\mathscr{C}$.
\end{proof}

\section{Congruences on the semigroups $\mathscr{B}(G)$ and
$\mathscr{B}^+(G)$}

The following lemma follows from the definition of a congruence on a
semilattice:

\begin{lemma}\label{lemma-2.4}
Let $\mathfrak{C}$ be an arbitrary congruence on a semilattice $S$
and let $\preccurlyeq$ be the natural partial order on $S$. Let $a$
and $b$ be idempotents of the semigroup $S$ such that
$a\mathfrak{C}b$. Then the relation $a\preccurlyeq b$ implies that
$a\mathfrak{C}c$ for all idempotents $c\in S$ such that
$a\preccurlyeq c\preccurlyeq b$.
\end{lemma}

A linearly ordered group $G$ is called \emph{archimedean} if for
each $a, b \in G^+\setminus\{e\}$ there exist positive integers $m$
and $n$ such that $b\leqslant a^m$ and $a\leqslant
b^n$~\cite{Birkhoff1973}. Linearly ordered archimedean groups may be
described as follows (\textbf{H\"{o}lder's Theorem}): \emph{A
linearly ordered group is Archimedean if and only if it is
isomorphic to some subgroup of the additive group of real numbers
with the natural order}~\cite{Holder1901}.

\begin{theorem}\label{theorem-2.5} Let $G$ be an archimedean linearly
ordered group. Then every non-trivial congruence on
$\mathscr{B}^+(G)$ is a group congruence.
\end{theorem}

\begin{proof}
Suppose that $\mathfrak{C}$ is a non-trivial congruence on the
semigroup $\mathscr{B}^+(G)$. Then there exist distinct elements
$\alpha^a_b$ and $\alpha^c_d$ of the semigroup $\mathscr{B}^+(G)$
such that $\alpha^a_b\mathfrak{C}\alpha^c_d$. Since by
Proposition~\ref{proposition-2.1}$(v)$ every $\mathscr{H}$-class of
the semigroup $\mathscr{B}^+(G)$ is a singleton set we conclude that
either $\alpha^a_b\cdot(\alpha^a_b)^{-1}\neq
\alpha^c_d\cdot(\alpha^c_d)^{-1}$ or
$(\alpha^a_b)^{-1}\cdot\alpha^a_b\neq
(\alpha^c_d)^{-1}\cdot\alpha^c_d$. We shall consider case
$\alpha^a_a=\alpha^a_b\cdot(\alpha^a_b)^{-1}\neq
\alpha^c_d\cdot(\alpha^c_d)^{-1}=\alpha^c_c$. In the other case the
proof is similar. Since by Proposition~\ref{proposition-2.1}$(ii)$
the band $E(\mathscr{B}^+(G))$ is a linearly ordered semilattice
without loss of generality we can assume that
$\alpha^c_c\preccurlyeq\alpha^a_a$. Then by
Proposition~\ref{proposition-2.1}$(i)$ we have that $a\leqslant c$
in $G$. Since $\alpha^a_b\mathfrak{C}\alpha^c_d$ and
$\mathscr{B}^+(G)$ is an inverse semigroup Lemma~III.1.1 from
\cite{Petrich1984} implies that
$\left(\alpha^a_b\cdot(\alpha^a_b)^{-1}\right)\mathfrak{C}
\left(\alpha^c_d\cdot(\alpha^c_d)^{-1}\right)$, i.e.,
$\alpha^a_a\mathfrak{C}\alpha^c_c$. Then we have that
\begin{eqnarray*}
  &&\alpha^c_a\cdot\alpha^a_a\cdot\alpha^a_c=\alpha^c_c; \\
  &&\alpha^c_a\cdot\alpha^c_c\cdot\alpha^a_c=\alpha^{c\cdot
  a^{-1}\cdot c}_{c\cdot a^{-1}\cdot c}; \\
  &&\alpha^c_a\cdot\alpha^{c\cdot a^{-1}\cdot c}_{c\cdot a^{-1}\cdot c}
  \cdot\alpha^a_c=
  \alpha^{c\cdot(a^{-1}\cdot c)^2}_{c\cdot(a^{-1}\cdot c)^2}; \\
  && \ldots \quad \ldots \\
  &&\alpha^c_a\cdot
  \alpha^{c\cdot(a^{-1}\cdot c)^{n-1}}_{c\cdot(a^{-1}\cdot c)^{n-1}}
  \cdot\alpha^a_c=
  \alpha^{c\cdot(a^{-1}\cdot c)^n}_{c\cdot(a^{-1}\cdot c)^n}.
\end{eqnarray*}
and hence $\alpha^a_a\mathfrak{C}\alpha^{c\cdot(a^{-1}\cdot
c)^n}_{c\cdot(a^{-1}\cdot c)^n}$ for every non-negative integer $n$.
Since $a<c$ in $G$ we get that $a^{-1}\cdot c$ is a positive element
of the linearly ordered group $G$. Since the linearly ordered group
$G$ is archimedean we conclude that for every $g\in G$ with $g>a$
there exists a positive integer $n$ such that $a^{-1}\cdot
g<\left(a^{-1}\cdot c\right)^{n}$ and hence $g<c\cdot
\left(a^{-1}\cdot c\right)^{n-1}$. Therefore Lemma~\ref{lemma-2.4}
and Proposition~\ref{proposition-2.1}$(i)$ imply that
$\alpha^a_a\mathfrak{C}\alpha^g_g$ for every $g\in G$ such that
$a\leqslant g$.

If $a=e$ then we have that all idempotents of the semigroup
$\mathscr{B}^+(G)$ are $\mathfrak{C}$-equivalent. Since the
semigroup $\mathscr{B}^+(G)$ is inverse we conclude that the
quotient semigroup $\mathscr{B}^+(G)/\mathfrak{C}$ contains only one
idempotent and by Lemma II.1.10 from \cite{Petrich1984} the
semigroup $\mathscr{B}^+(G)/\mathfrak{C}$ is a group.

Suppose that $e<a$. Then by Proposition~\ref{proposition-2.3} we
have that the semigroup $\mathscr{C}^*$ which is generated by
elements $\alpha^e_g$ and $\alpha^g_e$ is isomorphic to the bicyclic
semigroup for every element $g$ in $G^+$ such that $e<a\leqslant g$.
Hence we have that the following conditions hold :
\begin{equation*}
    \ldots\preccurlyeq \alpha^{g^i}_{g^i}
    \preccurlyeq \alpha^{g^{i-1}}_{g^{i-1}}\preccurlyeq
    \ldots\preccurlyeq\alpha^{g}_{g}\preccurlyeq\alpha^{a}_{a}
    \qquad \hbox{ and } \qquad
\end{equation*}
\begin{equation*}
     \alpha^{g^i}_{g^i}\neq
    \alpha^{g^{j}}_{g^{j}} \quad \hbox{ for distinct positive integers }
    i \hbox{ and } j,
\end{equation*}
in $E(\mathscr{B}^+(G))$. Since the linearly ordered group $G$ is
archimedean we conclude that
$\alpha^a_a\mathfrak{C}\alpha^{g^i}_{g^i}$ for every positive
integer $i$. Since the semigroup $\mathscr{C}^*$  is isomorphic to
the bicyclic semigroup we have that Corollary 1.32 of
\cite{CliffordPreston1961-1967} and Lemma~\ref{lemma-2.4} imply that
all idempotents of the semigroup $\mathscr{B}^+(G)$ are
$\mathfrak{C}$-equivalent. Since the semigroup $\mathscr{B}^+(G)$ is
inverse we conclude that the quotient semigroup
$\mathscr{B}^+(G)/\mathfrak{C}$ contains only one idempotent and by
Lemma II.1.10 from \cite{Petrich1984} the semigroup
$\mathscr{B}^+(G)/\mathfrak{C}$ is a group.
\end{proof}

\begin{theorem}\label{theorem-2.6} Let $G$ be an archimedean
linearly ordered group. Then every non-trivial congruence on
$\mathscr{B}(G)$ is a group congruence.
\end{theorem}

\begin{proof}
Suppose that $\mathfrak{C}$ is a non-trivial congruence on the
semigroup $\mathscr{B}(G)$. Similar arguments as in the proof of
Theorem~\ref{theorem-2.5} imply that there exist distinct
idempotents $\alpha^a_a$ and $\alpha^b_b$ in the semigroup
$\mathscr{B}(G)$ such that $\alpha^a_a\mathfrak{C}\alpha^b_b$ and
$\alpha^b_b\preccurlyeq\alpha^a_a$, for $a,b\in G$ with $a\leqslant
b$ in $G$. Then we have that
\begin{equation*}
    \alpha^e_a\cdot\alpha^a_a\cdot\alpha^a_e=\alpha^e_e \qquad
    \hbox{ and } \qquad
    \alpha^e_a\cdot\alpha^b_b\cdot\alpha^a_e=
    \alpha^{b\cdot a^{-1}}_b\cdot\alpha^a_e=
    \alpha^{b\cdot a^{-1}}_{b\cdot a^{-1}},
\end{equation*}
and hence $\alpha^e_e\mathfrak{C}\alpha^{b\cdot a^{-1}}_{b\cdot
a^{-1}}$. Since $a\leqslant b$ in $G$ we conclude that $e\leqslant
b\cdot a^{-1}$ in $G$ and hence Theorem~\ref{theorem-2.5} implies
that $\alpha^c_c\mathfrak{C}\alpha^d_d$ for all $c,d\in G^+$.

We fix an arbitrary element $g\in G\setminus G^+$. Then we have that
$g^{-1}\in G^+\setminus\{e\}$ and hence
$\alpha^e_e\mathfrak{C}\alpha^{g^{-1}}_{g^{-1}}$. Since
\begin{equation*}
  \alpha^g_e\cdot\alpha^e_e\cdot\alpha^e_g=\alpha^g_g \quad
    \hbox{ and } \quad
  \alpha^g_e\cdot\alpha^{g^{-1}}_{g^{-1}}\cdot\alpha^e_g=
  \alpha^{g^{-1}\cdot e\cdot g}_{g^{-1}}\cdot\alpha^e_g=
  \alpha^{e}_{g^{-1}}\cdot\alpha^e_g=
  \alpha^{e}_{g^{-1}\cdot e\cdot g}=
  \alpha^e_e
\end{equation*}
we conclude that $\alpha^e_e\mathfrak{C}\alpha^{g}_{g}$. Therefore
all idempotents of the semigroup $\mathscr{B}(G)$ are
$\mathfrak{C}$-equivalent. Since the semigroup $\mathscr{B}(G)$ is
inverse we conclude that the quotient semigroup
$\mathscr{B}(G)/\mathfrak{C}$ contains only one idempotent and by
Lemma II.1.10 from \cite{Petrich1984} the semigroup
$\mathscr{B}(G)/\mathfrak{C}$ is a group.
\end{proof}

\begin{remark}\label{remark-2.7}
We observe that Proposition~\ref{proposition-1.4} implies that if
$G$ is a linearly ordered $d$-group then the statements similar to
Propositions~\ref{proposition-2.1} and \ref{proposition-2.3} and
Theorems~\ref{theorem-2.5} and \ref{theorem-2.6} hold for the
semigroups $\overset{_\circ}{\mathscr{B}}(G)$ and
$\overset{_\circ}{\mathscr{B}}{}^+(G)$.
\end{remark}

\begin{theorem}\label{theorem-2.12a}
If $G$ is the lexicographic product $A\times_\textbf{lex}H$ of
non-singleton linearly ordered groups $A$ and $H$, then the
semigroups $\mathscr{B}(G)$ and $\mathscr{B}^+(G)$ have non-trivial
non-group congruences.
\end{theorem}

\begin{proof}
We define a relation $\sim_{\mathfrak{c}}$ on the semigroup
$\mathscr{B}(G)$ as follows:
\begin{equation*}
    \alpha^{(a_1,b_1)}_{(c_1,d_1)}\sim_{\mathfrak{c}}
    \alpha^{(a_2,b_2)}_{(c_2,d_2)} \quad \hbox{if and only if}
    \quad
    a_1=a_2, \,  c_1=c_2 \quad  \hbox{and}
    \quad  d^{-1}_1b_1=d_2^{-1}b_2.
\end{equation*}
Simple verifications show that $\sim_{\mathfrak{c}}$ is an
equivalence relation on the semigroup $\mathscr{B}(G)$.

Next we shall prove that $\sim_{\mathfrak{c}}$ is a congruence on
$\mathscr{B}(G)$. Suppose that
$\alpha^{(a_1,b_1)}_{(c_1,d_1)}\sim_{\mathfrak{c}}
\alpha^{(a_2,b_2)}_{(c_2,d_2)}$ for some
$\alpha^{(a_1,b_1)}_{(c_1,d_1)},\alpha^{(a_2,b_2)}_{(c_2,d_2)}\in
\mathscr{B}(G)$. Let $\alpha^{(u,v)}_{(x,y)}$ be an arbitrary
element of $\mathscr{B}(G)$. Then we have that
\begin{equation*}
\begin{split}
    \alpha^{(k_1,l_1)}_{(m_1,n_1)}=&\,
    \alpha^{(a_1,b_1)}_{(c_1,d_1)}\cdot \alpha^{(u,v)}_{(x,y)}=
\left\{
  \begin{array}{ll}
    \alpha^{(u,v)\cdot(c_1,d_1)^{-1}\cdot(a_1,b_1)}_{(x,y)}\!, &
     \hbox{if~} (c_1,d_1){\leqslant}(u,v); \\
    \alpha^{(a_1,b_1)}_{(c_1,d_1)\cdot(u,v)^{-1}\cdot(x,y)}, &
     \hbox{if~} (u,v){\leqslant}(c_1,d_1)
  \end{array}
\right.=\\
= &
 \left\{
  \begin{array}{ll}
    \alpha^{(uc_1^{-1}a_1,vd_1^{-1}b_1)}_{(x,y)}\!, &
     \hbox{if~} (c_1,d_1){\leqslant}(u,v); \\
    \alpha^{(a_1,b_1)}_{(c_1u^{-1}x,d_1v^{-1}y)}, &
     \hbox{if~} (u,v){\leqslant}(c_1,d_1)
  \end{array}
\right.=\\
= &
 \left\{
  \begin{array}{ll}
    \alpha^{(uc_1^{-1}a_1,vd_1^{-1}b_1)}_{(x,y)}\!, &
     \hbox{if~} c_1<u; \\
    \alpha^{(a_1,vd_1^{-1}b_1)}_{(x,y)}\!, &
     \hbox{if~} c_1=u \hbox{ and } d_1\leqslant v; \\
    \alpha^{(a_1,b_1)}_{(c_1u^{-1}x,d_1v^{-1}y)}, &
     \hbox{if~} u<c_1;\\
    \alpha^{(a_1,b_1)}_{(x,d_1v^{-1}y)}, &
     \hbox{if~} u=c_1 \hbox{ and } v\leqslant d_1;
  \end{array}
\right.
\end{split}
\end{equation*}
and
\begin{equation*}
\begin{split}
    \alpha^{(k_2,l_2)}_{(m_2,n_2)}= &\,
    \alpha^{(a_2,b_2)}_{(c_2,d_2)}\cdot \alpha^{(u,v)}_{(x,y)}=
\left\{
  \begin{array}{ll}
    \alpha^{(u,v)\cdot(c_2,d_2)^{-1}\cdot(a_2,b_2)}_{(x,y)}\!, &
     \hbox{if~} (c_2,d_2){\leqslant}(u,v); \\
    \alpha^{(a_2,b_2)}_{(c_2,d_2)\cdot(u,v)^{-1}\cdot(x,y)}, &
     \hbox{if~} (u,v){\leqslant}(c_2,d_2)
  \end{array}
\right.=\\
 =& \left\{
  \begin{array}{ll}
    \alpha^{(uc_2^{-1}a_2,vd_2^{-1}b_2)}_{(x,y)}, &
     \hbox{if~} (c_2,d_2){\leqslant}(u,v); \\
    \alpha^{(a_2,b_2)}_{(c_2u^{-1}x,d_2v^{-1}y)}, &
     \hbox{if~} (u,v){\leqslant}(c_2,d_2),
  \end{array}
\right.=\\
= &
 \left\{
  \begin{array}{ll}
    \alpha^{(uc_2^{-1}a_2,vd_2^{-1}b_2)}_{(x,y)}\!, &
     \hbox{if~} c_2<u; \\
    \alpha^{(a_2,vd_2^{-1}b_2)}_{(x,y)}\!, &
     \hbox{if~} c_2=u \hbox{ and } d_2\leqslant v; \\
    \alpha^{(a_2,b_2)}_{(c_2u^{-1}x,d_2v^{-1}y)}, &
     \hbox{if~} u<c_2;\\
    \alpha^{(a_2,b_2)}_{(x,d_2v^{-1}y)}, &
     \hbox{if~} u=c_2 \hbox{ and } v\leqslant d_2;
  \end{array}
\right.
\end{split}
\end{equation*}
and since $a_1=a_2$,  $c_1=c_2$ and $d^{-1}_1b_1=d_2^{-1}b_2$ we
conclude that the following conditions hold:
\begin{itemize}
  \item[$(1)$] if $c_1=c_2<u$, then $k_1=uc_1^{-1}a_1=uc_2^{-1}a_2=k_2$,
   $m_1=x=m_2$ and
   \begin{equation*}
    n^{-1}_1l_1=y^{-1}vd_1^{-1}b_1=y^{-1}vd_2^{-1}b_2=n^{-1}_2l_2;
   \end{equation*}
  \item[$(2)$] if $c_1=c_2=u$ and $d_1\leqslant v$, then
   $k_1=a_1=a_2=k_2$, $m_1=x=m_2$ and
   \begin{equation*}
    n^{-1}_1l_1=y^{-1}vd_1^{-1}b_1=y^{-1}vd_2^{-1}b_2=n^{-1}_2l_2;
   \end{equation*}
  \item[$(3)$] if $u<c_1=c_2$, then $k_1=a_1=a_2=k_2$,
   $m_1=c_1u^{-1}x=c_2u^{-1}x=m_2$ and
   \begin{equation*}
    n^{-1}_1l_1=y^{-1}vd_1^{-1}b_1=y^{-1}vd_2^{-1}b_2=n^{-1}_2l_2;
   \end{equation*}
  \item[$(4)$] if $u=c_1=c_2$ and $v\leqslant d_1$, then
   $k_1=a_1=a_2=k_2$, $m_1=x=m_2$ and
   \begin{equation*}
    n^{-1}_1l_1=y^{-1}vd_1^{-1}b_1=y^{-1}vd_2^{-1}b_2=n^{-1}_2l_2.
   \end{equation*}
\end{itemize}
Hence we get that $\alpha^{(k_1,l_1)}_{(m_1,n_1)}
\sim_{\mathfrak{c}} \alpha^{(k_2,l_2)}_{(m_2,n_2)}$. Similarly we
have that
\begin{equation*}
\begin{split}
    \alpha^{(p_1,q_1)}_{(r_1,s_1)}= &\,
    \alpha^{(u,v)}_{(x,y)}\cdot \alpha^{(a_1,b_1)}_{(c_1,d_1)}=
\left\{
  \begin{array}{ll}
    \alpha^{(a_1,b_1)\cdot(x,y)^{-1}\cdot(u,v)}_{(c_1,d_1)}\!, &
     \hbox{if~} (x,y){\leqslant}(a_1,b_1); \\
    \alpha^{(u,v)}_{(x,y)\cdot(a_1,b_1)^{-1}\cdot(c_1,d_1)}, &
     \hbox{if~} (a_1,b_1){\leqslant}(x,y)
  \end{array}
\right.=\\
 = &\left\{
  \begin{array}{ll}
    \alpha^{(a_1x^{-1}u,b_1y^{-1}v)}_{(c_1,d_1)}\!, &
     \hbox{if~} (x,y){\leqslant}(a_1,b_1); \\
    \alpha^{(u,v)}_{(xa_1^{-1}c_1,yb_1^{-1}d_1)}, &
     \hbox{if~} (a_1,b_1){\leqslant}(x,y),
  \end{array}
\right.=\\
 = &\left\{
  \begin{array}{ll}
    \alpha^{(a_1x^{-1}u,b_1y^{-1}v)}_{(c_1,d_1)}\!, &
     \hbox{if~} x<a_1; \\
    \alpha^{(u,b_1y^{-1}v)}_{(c_1,d_1)}\!, &
     \hbox{if~} x=a_1 \hbox{ and } y\leqslant b_1; \\
    \alpha^{(u,v)}_{(xa_1^{-1}c_1,yb_1^{-1}d_1)}, &
     \hbox{if~} a_1<x;\\
    \alpha^{(u,v)}_{(c_1,yb_1^{-1}d_1)}, &
     \hbox{if~} a_1=x \hbox{ and } b_1\leqslant y,
  \end{array}
\right.
\end{split}
\end{equation*}
and
\begin{equation*}
\begin{split}
    \alpha^{(p_2,q_2)}_{(r_2,s_2)}= &\,
    \alpha^{(u,v)}_{(x,y)}\cdot \alpha^{(a_2,b_2)}_{(c_2,d_2)}=
\left\{
  \begin{array}{ll}
    \alpha^{(a_2,b_2)\cdot(x,y)^{-1}\cdot(u,v)}_{(c_2,d_2)}\!, &
     \hbox{if~} (x,y){\leqslant}(a_2,b_2); \\
    \alpha^{(u,v)}_{(x,y)\cdot(a_2,b_2)^{-1}\cdot(c_2,d_2)}, &
     \hbox{if~} (a_2,b_2){\leqslant}(x,y)
  \end{array}
\right.=\\
 =& \left\{
  \begin{array}{ll}
    \alpha^{(a_2x^{-1}u,b_2y^{-1}v)}_{(c_2,d_2)}\!, &
     \hbox{if~} (x,y){\leqslant}(a_2,b_2); \\
    \alpha^{(u,v)}_{(xa_2^{-1}c_2,yb_2^{-1}d_2)}, &
     \hbox{if~} (a_2,b_2){\leqslant}(x,y)
  \end{array}
\right.=\\
 = &\left\{
  \begin{array}{ll}
    \alpha^{(a_2x^{-1}u,b_2y^{-1}v)}_{(c_2,d_2)}\!, &
     \hbox{if~} x<a_2; \\
    \alpha^{(u,b_2y^{-1}v)}_{(c_2,d_2)}\!, &
     \hbox{if~} x=a_2 \hbox{ and } y\leqslant b_2; \\
    \alpha^{(u,v)}_{(xa_2^{-1}c_2,yb_2^{-1}d_2)}, &
     \hbox{if~} a_2<x;\\
    \alpha^{(u,v)}_{(c_2,yb_2^{-1}d_2)}, &
     \hbox{if~} a_2=x \hbox{ and } b_2\leqslant y,
  \end{array}
\right.
\end{split}
\end{equation*}
and since $a_1=a_2$,  $c_1=c_2$ and $d^{-1}_1b_1=d_2^{-1}b_2$ we
conclude that the following conditions hold:
\begin{itemize}
  \item[$(1)$] if $x<a_1=a_2$, then $p_1=a_1x^{-1}u=a_2x^{-1}u=p_2$,
   $r_1=c_1=c_2=r_2$ and
   \begin{equation*}
    s^{-1}_1q_1=d_1^{-1}b_1y^{-1}v=d_2^{-1}b_2y^{-1}v=s^{-1}_2q_2;
   \end{equation*}
  \item[$(2)$] if $x=a_1=a_2$ and $y\leqslant b_1$,
   then $p_1=u=p_2$,
   $r_1=c_1=c_2=r_2$ and
   \begin{equation*}
    s^{-1}_1q_1=d_1^{-1}b_1y^{-1}v=d_2^{-1}b_2y^{-1}v=s^{-1}_2q_2;
   \end{equation*}
  \item[$(3)$] if $a_1=a_2<x$, then $p_1=u=p_2$,
   $r_1=xa_1^{-1}c_1=xa_2^{-1}c_2=r_2$ and
   \begin{equation*}
    s^{-1}_1q_1=d_1^{-1}b_1y^{-1}v=d_2^{-1}b_2y^{-1}v=s^{-1}_2q_2;
   \end{equation*}
  \item[$(4)$] if $a_1=a_2=x$ and $b_1\leqslant y$, then $p_1=u=p_2$,
   $r_1=c_1=c_2=r_2$ and
   \begin{equation*}
    s^{-1}_1q_1=d_1^{-1}b_1y^{-1}v=d_2^{-1}b_2y^{-1}v=s^{-1}_2q_2.
   \end{equation*}
\end{itemize}
Hence we get that $\alpha^{(p_1,q_1)}_{(r_1,s_1)}
\sim_{\mathfrak{c}} \alpha^{(p_2,q_2)}_{(r_2,s_2)}$.

We fix any $a_1,a_2,b_1,b_2\in G$. If $a_1\neq a_2$ then we have
that the elements $\alpha^{(a_1,b_1)}_{(a_1,b_1)}$ and
$\alpha^{(a_2,b_2)}_{(a_2,b_2)}$ are idempotents of the semigroup
$\mathscr{B}(G)$, and moreover the elements
$\alpha^{(a_1,b_1)}_{(a_1,b_1)}$ and
$\alpha^{(a_2,b_2)}_{(a_2,b_2)}$ are not $\sim_c$-equivalent. Since
a homomorphic image of an idempotent is an idempotent too, we
conclude that $\left(\alpha^{(a_1,b_1)}_{(a_1,b_1)}\right)\pi_c
\neq\left(\alpha^{(a_2,b_2)}_{(a_2,b_2)}\right)\pi_c$, where
$\pi_c\colon\mathscr{B}(G)\rightarrow\mathscr{B}(G)/\!\sim_c$ is the
natural homomorphism which is generated by the congruence $\sim_c$
on the semigroup $\mathscr{B}(G)$. This implies that the quotient
semigroup $\mathscr{B}(G)/\!\sim_c$ is not a group, and hence
$\sim_c$ is not a group congruence on the semigroup on the semigroup
$\mathscr{B}(G)$.

The proof of the statement that the semigroup $\mathscr{B}^+(G)$ has
a non-trivial non-group congruence is similar.
\end{proof}

\begin{theorem}\label{theorem-2.12b}
Let $G$ be a commutative linearly ordered group. Then the following
conditions are equivalent:
\begin{itemize}
  \item[$(i)$] $G$ is archimedean;
  \item[$(ii)$] every non-trivial congruence on
   $\mathscr{B}(G)$ is a group congruence; \; and
  \item[$(iii)$] every non-trivial congruence on
   $\mathscr{B}^+(G)$ is a group congruence.
\end{itemize}
\end{theorem}

\begin{proof}
Implications $(i)\Rightarrow(ii)$ and $(i)\Rightarrow(iii)$ follow
from Theorems~\ref{theorem-2.6} and \ref{theorem-2.5}, respectively.

$(ii)\Rightarrow(i)$ Suppose the contrary that there exists a
non-archimedean commutative linearly ordered group $G$ such that
every non-trivial congruence on $\mathscr{B}(G)$ is a group
congruence. Then by Hahn Theorem (see \cite{Hahn1907} or
\cite[Section~VII.3, Theorem~1]{KokorinKopytov1972}) $G$ is
isomorphic to a lexicographic product
$\displaystyle{\prod_{\alpha\in\mathscr{J}}}
{}_\textbf{lex}H_{\alpha}$ of some family of non-singleton subgroups
$\{H_{\alpha}\mid\alpha\in\mathscr{J}\}$ of the additive group of
real numbers with a non-singleton linearly ordered index set
$\mathscr{J}$. We fix a non-maximal element
$\alpha_0\in\mathscr{J}$, and put
\begin{equation*}
A=\prod{}_\textbf{lex}\{H_{\alpha}\mid\alpha\leqslant\alpha_0\}
\qquad \hbox{ and } \qquad
H=\prod{}_\textbf{lex}\{H_{\alpha}\mid\alpha_0<\alpha\}.
\end{equation*}
Then we have that $G$ is isomorphic to a lexicographic product
$A\times_\textbf{lex}H$ of non-singleton linearly ordered groups $A$
and $H$, and hence by Theorem~\ref{theorem-2.12a} the semigroup
$\mathscr{B}(G)$ has a non-trivial non-group congruence. The
obtained contradiction implies that the group $G$ is archimedean.

The proof of implication $(iii)\Rightarrow(i)$ is similar to
$(ii)\Rightarrow(i)$.
\end{proof}

On the semigroup $\overline{\mathscr{B}}(G)$ (resp.,
$\overline{\mathscr{B}}\,^+(G)$) we determine a relation
$\sim_{\mathfrak{id}}$ in the following way. We define a map
$\mathfrak{id}\colon\overline{\mathscr{B}}(G)\rightarrow
\overline{\mathscr{B}}(G)$ (resp., $\mathfrak{id}\colon
\overline{\mathscr{B}}\,^+(G)\rightarrow
\overline{\mathscr{B}}\,^+(G)$) by the formulae
$(\alpha^g_h)\mathfrak{id}=\overset{_\circ}{\alpha}{}^g_h$ and
$(\overset{_\circ}{\alpha}{}^g_h)\mathfrak{id}=\alpha^g_h$ for
$g,h\in G$ (resp., $g,h\in G^+$). We put
\begin{equation*}
    \alpha\sim_{\mathfrak{id}}\beta \quad \hbox{if and only if}
    \quad \alpha=\beta \; \hbox{ or } \;
    (\alpha)\mathfrak{id}=\beta \; \hbox{ or } \;
    (\beta)\mathfrak{id}=\alpha,
\end{equation*}
$\hbox{for } \;
    \alpha,\beta\in\overline{\mathscr{B}}(G)
    \;  (\hbox{resp.,~}\alpha,\beta\in\overline{\mathscr{B}}\,^+(G)).$
Simple verifications show that $\sim_{\mathfrak{id}}$ is an
equivalence relation on the semigroup $\overline{\mathscr{B}}(G)$
(resp., $\overline{\mathscr{B}}\,^+(G)$).

\begin{proposition}\label{proposition-2.10}
If $G$ is a linearly ordered $d$-group then $\sim_{\mathfrak{id}}$
is a congruence on semigroups $\overline{\mathscr{B}}(G)$ and
$\overline{\mathscr{B}}\,^+(G)$. Moreover, quotient semigroups
$\overline{\mathscr{B}}(G)/\!\sim_{\mathfrak{id}}$ and
$\overline{\mathscr{B}}(G)\,^+/\!\sim_{\mathfrak{id}}$ are
isomorphic to semigroups $\mathscr{B}(G)$ and $\mathscr{B}\,^+(G)$,
respectively.
\end{proposition}

\begin{proof}
It is sufficient to show that if $\alpha\sim_{\mathfrak{id}}\beta$
and $\gamma\sim_{\mathfrak{id}}\delta$ then
$(\alpha\cdot\gamma)\sim_{\mathfrak{id}}(\beta\cdot\delta)$ for
$\alpha,\beta,\gamma,\delta\in\overline{\mathscr{B}}(G)$ (resp.,
$\alpha,\beta,\gamma,\delta\in\overline{\mathscr{B}}\,^+(G)$). Since
the case $\alpha=\beta$ and $\gamma=\delta$ is trivial we consider
the following cases:
\begin{itemize}
  \item[$(i)$] $\alpha=\alpha^a_b$,
       $\beta=\overset{_\circ}{\alpha}{}^a_b$ and
       $\gamma=\delta=\alpha^c_d$;

  \item[$(ii)$] $\alpha=\alpha^a_b$,
       $\beta=\overset{_\circ}{\alpha}{}^a_b$ and
       $\gamma=\delta=\overset{_\circ}{\alpha}{}^c_d$;

  \item[$(iii)$] $\alpha=\overset{_\circ}{\alpha}{}^a_b$,
       $\beta=\alpha^a_b$ and
       $\gamma=\delta=\alpha^c_d$;

  \item[$(iv)$] $\alpha=\overset{_\circ}{\alpha}{}^a_b$,
       $\beta=\alpha^a_b$ and
       $\gamma=\delta=\overset{_\circ}{\alpha}{}^c_d$;

  \item[$(v)$] $\alpha=\alpha^a_b$,
       $\beta=\overset{_\circ}{\alpha}{}^a_b$,
       $\gamma=\overset{_\circ}{\alpha}{}^c_d$ and
       $\delta=\alpha^c_d$;

  \item[$(vi)$] $\alpha=\alpha^a_b$,
       $\beta=\overset{_\circ}{\alpha}{}^a_b$, $\gamma=\alpha^c_d$
       and  $\delta=\overset{_\circ}{\alpha}{}^c_d$;

  \item[$(vii)$] $\alpha=\overset{_\circ}{\alpha}{}^a_b$,
       $\beta=\alpha^a_b$,
       $\gamma=\overset{_\circ}{\alpha}{}^c_d$ and
       $\delta=\alpha^c_d$;

  \item[$(viii)$] $\alpha=\overset{_\circ}{\alpha}{}^a_b$,
       $\beta=\alpha^a_b$,
       $\gamma=\alpha^c_d$ and
       $\delta=\overset{_\circ}{\alpha}{}^c_d$;

  \item[$(ix)$] $\alpha=\beta=\alpha^a_b$,
       $\gamma=\overset{_\circ}{\alpha}{}^c_d$ and
       $\delta=\alpha^c_d$;

  \item[$(x)$] $\alpha=\beta=\overset{_\circ}{\alpha}{}^a_b$,
       $\gamma=\overset{_\circ}{\alpha}{}^c_d$ and
       $\delta=\alpha^c_d$;

  \item[$(xi)$] $\alpha=\beta=\alpha^a_b$,
       $\gamma=\alpha^c_d$ and
       $\delta=\overset{_\circ}{\alpha}{}^c_d$;\; and

  \item[$(xii)$] $\alpha=\beta=\overset{_\circ}{\alpha}{}^a_b$,
       $\gamma=\alpha^c_d$ and
       $\delta=\overset{_\circ}{\alpha}{}^c_d$,
\end{itemize}
where $a,b,c,d\in G$ (resp., $a,b,c,d\in G^+$).

In case $(i)$ we have that
\begin{equation*}
    \alpha\cdot\gamma=\alpha^a_b\cdot\alpha^c_d{=}\!
\left\{
  \begin{array}{ll}
    \alpha^{c\cdot b^{-1}\cdot a}_d, & \hbox{if } b<c; \\
    \alpha^a_d, & \hbox{if } b=c; \\
    \alpha^a_{b\cdot c^{-1}\cdot d}, & \hbox{if } b>c,
  \end{array}
\right. \, \hbox{and } \,
    \beta\cdot\delta=\overset{_\circ}{\alpha}{}^a_b\cdot\alpha^c_d{=}\!
\left\{
  \begin{array}{ll}
    \overset{_\circ}{\alpha}{}^{c\cdot b^{-1}\cdot a}_d,
      & \hbox{if } b<c; \\
    \overset{_\circ}{\alpha}{}^a_d, & \hbox{if } b=c; \\
    \alpha^a_{b\cdot c^{-1}\cdot d}, & \hbox{if } b>c,
  \end{array}
\right.
\end{equation*}
and hence
$(\alpha\cdot\gamma)\sim_{\mathfrak{id}}(\beta\cdot\delta)$ in
$\overline{\mathscr{B}}(G)$ (resp.,
$\overline{\mathscr{B}}\,^+(G)$). In other cases verifications are
similar.

Since the restriction
$\Phi_{\sim_{\mathfrak{id}}}|_{\mathscr{B}(G)}\colon\mathscr{B}(G)
\rightarrow\mathscr{B}(G)$ of the natural homomorphism
$\Phi_{\sim_{\mathfrak{id}}}\colon\overline{\mathscr{B}}(G)
\rightarrow\mathscr{B}(G)$ is a bijective map we conclude that the
semigroup $(\overline{\mathscr{B}}(G))\Phi_{\sim_{\mathfrak{id}}}$
is isomorphic to the semigroup $\mathscr{B}(G)$. Similar arguments
show that the semigroup
$\overline{\mathscr{B}}\,^+(G)/\!\sim_{\mathfrak{id}}$ is isomorphic
to $\mathscr{B}\,^+(G)$.
\end{proof}

\begin{theorem}\label{theorem-2.11} Let $G$ be an archimedean
linearly ordered $d$-group. If $\mathfrak{C}$ is a non-trivial
congruence on $\overline{\mathscr{B}}(G)$ $($resp.,
$\overline{\mathscr{B}}\,^+(G))$ then the quotient semigroup
$\overline{\mathscr{B}}(G)/\mathfrak{C}$ $($resp.,
$\overline{\mathscr{B}}\,^+(G)/\mathfrak{C})$ is either a group or
$\overline{\mathscr{B}}(G)/\mathfrak{C}$ $($resp.,
$\overline{\mathscr{B}}\,^+(G)/\mathfrak{C})$ is isomorphic to the
semigroup $\mathscr{B}(G)$ $($resp., $\mathscr{B}\,^+(G))$.
\end{theorem}

\begin{proof}
Since the subsemigroup of idempotents of the semigroup
$\overline{\mathscr{B}}(G)$ is linearly ordered we have that similar
arguments as in the proof of Theorem~\ref{theorem-2.5} imply that
there exist distinct idempotents $\varepsilon$ and $\iota$ of
$\overline{\mathscr{B}}(G)$ such that $\varepsilon\mathfrak{C}\iota$
and $\varepsilon\preccurlyeq\iota$. If the set
$(\varepsilon,\iota)=\{\upsilon\in E(\overline{\mathscr{B}}(G))\mid
\varepsilon\prec\upsilon\prec\iota\}$ is non-empty then
Lemma~\ref{lemma-2.4} and Theorem~\ref{theorem-2.5} imply that the
quotient semigroup $\overline{\mathscr{B}}(G)/\mathfrak{C}$ is
inverse and has only one idempotent, and hence by Lemma~II.1.10 from
\cite{Petrich1984} it is a group. Otherwise there exists $g\in G$
such that $\iota=\alpha^g_g$ and
$\varepsilon=\overset{_\circ}{\alpha}{}^g_g$. Since
$\alpha^k_l=\alpha^k_g\cdot\alpha^g_g\cdot\alpha^g_l$ and
$\overset{_\circ}{\alpha}{}^k_l=\alpha^k_g\cdot
\overset{_\circ}{\alpha}{}^g_g\cdot\alpha^g_l$ for every $k,l\in G$
we concludes that the congruence $\mathfrak{C}$ coincides with the
congruence $\sim_{\mathfrak{id}}$ on $\overline{\mathscr{B}}(G)$,
and hence by Proposition~\ref{proposition-2.10} the quotient
semigroup  $\overline{\mathscr{B}}(G)/\mathfrak{C}$ is isomorphic to
the semigroup $\mathscr{B}(G)$.

In the case of the semigroup $\overline{\mathscr{B}}\,^+(G)$ the
proof is similar.
\end{proof}

\begin{theorem}\label{theorem-2.13b}
Let $G$ be a commutative linearly ordered $d$-group. Then the
following conditions are equivalent:
\begin{itemize}
  \item[$(i)$] $G$ is archimedean;
  \item[$(ii)$] every non-trivial congruence on
   $\overset{_\circ}{\mathscr{B}}(G)$ is a group congruence; \; and
  \item[$(iii)$] every non-trivial congruence on
   $\overset{_\circ}{\mathscr{B}}{}^+(G)$ is a group congruence;
  \item[$(iv)$] the semigroup $\overline{\mathscr{B}}(G)$ has a
   unique non-trivial non-group congruence;
  \item[$(v)$] the semigroup $\overline{\mathscr{B}}\,^+(G)$ has a
   unique non-trivial non-group congruence.
\end{itemize}
\end{theorem}

\begin{proof}
The equivalence of statements $(i)$, $(ii)$ and $(iii)$ follows from
Proposition~\ref{proposition-1.4} and Theorem~\ref{theorem-2.12b}.
Also Theorem~\ref{theorem-2.11} implies that  implications
$(i)\Rightarrow(iv)$ and $(i)\Rightarrow(v)$ hold.

Next we shall show that implication $(iv)\Rightarrow(i)$ holds.
Suppose the contrary: there exists a commutative linearly ordered
non-archimedean $d$-group $G$ such that the semigroup
$\overline{\mathscr{B}}(G)$ has a unique non-trivial non-group
congruence. Then by Proposition~\ref{proposition-2.10} we have that
$\sim_{\mathfrak{id}}$ is a unique non-trivial non-group congruence
on the semigroup $\overline{\mathscr{B}}(G)$. Therefore, similarly
as in the proof of Theorem~\ref{theorem-2.12b} we get that $G$ is
isomorphic to the lexicographic product $A\times_\textbf{lex}H$ of
non-singleton linearly ordered groups $A$ and $H$, and hence by
Theorem~\ref{theorem-2.12a} we have that the semigroup
$\mathscr{B}(G)$ has a non-trivial non-group congruence $\sim$. We
define a relation $\overline{\sim}$ on the semigroup
$\overline{\mathscr{B}}(G)$ as follows:
\begin{itemize}
  \item[$(i)$]
   $\left(\alpha^{(a,b)}_{(c,d)},\alpha^{(p,q)}_{(r,s)}\right)
   \in\overline{\sim}$ if and only if
   $\left(\alpha^{(a,b)}_{(c,d)},\alpha^{(p,q)}_{(r,s)}\right)
   \in\sim$, for $\alpha^{(a,b)}_{(c,d)},\alpha^{(p,q)}_{(r,s)}\in
   \mathscr{B}(G)\subset\overline{\mathscr{B}}(G)$;

  \item[$(ii)$]
   $\left(\alpha^{(p,q)}_{(r,s)},
   \overset{_\circ}{\alpha}{}^{(p,q)}_{(r,s)}\right),
   \left(\overset{_\circ}{\alpha}{}^{(p,q)}_{(r,s)},
   \alpha^{(p,q)}_{(r,s)}\right),
   \left(\overset{_\circ}{\alpha}{}^{(p,q)}_{(r,s)},
   \overset{_\circ}{\alpha}{}^{(p,q)}_{(r,s)}\right)
   \in\overline{\sim}$, for all $p,r\in A$ and $q,s\in H$;

  \item[$(iii)$]
   $\left(\overset{_\circ}{\alpha}{}^{(a,b)}_{(c,d)},
   \overset{_\circ}{\alpha}{}^{(p,q)}_{(r,s)}\right)
   \in\overline{\sim}$ if and only if
   $\left(\alpha^{(a,b)}_{(c,d)},\alpha^{(p,q)}_{(r,s)}\right)
   \in\sim$, for $\alpha^{(a,b)}_{(c,d)},
   \alpha^{(p,q)}_{(r,s)}\in
   \mathscr{B}(G)\subset\overline{\mathscr{B}}(G)$ and
   $\overset{_\circ}{\alpha}{}^{(a,b)}_{(c,d)},
   \overset{_\circ}{\alpha}{}^{(p,q)}_{(r,s)}\in
   \overset{_\circ}{\mathscr{B}}(G)\subset\overline{\mathscr{B}}(G)$;

  \item[$(iv)$]
   $\left(\overset{_\circ}{\alpha}{}^{(a,b)}_{(c,d)},
   \alpha^{(p,q)}_{(r,s)}\right)
   \in\overline{\sim}$ if and only if
   $\left(\alpha^{(a,b)}_{(c,d)},\alpha^{(p,q)}_{(r,s)}\right)
   \in\sim$, for $\alpha^{(a,b)}_{(c,d)},
   \alpha^{(p,q)}_{(r,s)}\in
   \mathscr{B}(G)\subset\overline{\mathscr{B}}(G)$ and
   $\overset{_\circ}{\alpha}{}^{(a,b)}_{(c,d)}\in
   \overset{_\circ}{\mathscr{B}}(G)\subset\overline{\mathscr{B}}(G)$;

  \item[$(v)$]
   $\left(\alpha^{(a,b)}_{(c,d)},
   \overset{_\circ}{\alpha}{}^{(p,q)}_{(r,s)}\right)
   \in\overline{\sim}$ if and only if
   $\left(\alpha^{(a,b)}_{(c,d)},\alpha^{(p,q)}_{(r,s)}\right)
   \in\sim$, for $\alpha^{(a,b)}_{(c,d)},
   \alpha^{(p,q)}_{(r,s)}\in
   \mathscr{B}(G)\subset\overline{\mathscr{B}}(G)$ and
   $\overset{_\circ}{\alpha}{}^{(p,q)}_{(r,s)}\in
   \overset{_\circ}{\mathscr{B}}(G)\subset\overline{\mathscr{B}}(G)$.
\end{itemize}

Then simple verifications show that $\overline{\sim}$ is a
congruence on the semigroup $\overline{\mathscr{B}}(G)$, and
moreover the quotient semigroup
$\overline{\mathscr{B}}(G)/{\overline{\sim}}$ is isomorphic to the
quotient semigroup $\mathscr{B}(G)/\sim$. This implies that the
congruence $\overline{\sim}$ is different from
$\sim_{\mathfrak{id}}$. This contradicts that $\sim_{\mathfrak{id}}$
is a unique non-trivial non-group congruence on the semigroup
$\overline{\mathscr{B}}(G)$. The obtained contradiction implies
implication $(iv)\Rightarrow(i)$.

The proof of implication $(v)\Rightarrow(i)$ is similar to
implication $(iv)\Rightarrow(i)$.
\end{proof}

\begin{theorem}\label{theorem-2.13}
Let $G$ be a linearly ordered group and $\mathfrak{C}_{mg}$ be the
least group congruence on the semigroup $\mathscr{B}(G)$
$(\mbox{resp., } \mathscr{B}^+(G))$. Then the quotient semigroup
$\mathscr{B}(G)/{\mathfrak{C}_{mg}}$ $(\mbox{resp., }
\mathscr{B}^+(G)/{\mathfrak{C}_{mg}})$ is antiisomorphic to the
group $G$.
\end{theorem}

\begin{proof}
By Proposition~\ref{proposition-1.1}$(ii)$ and Lemma~III.5.2 from
\cite{Petrich1984} we have that elements $\alpha^a_b$ and
$\alpha^c_d$ are $\mathfrak{C}_{mg}$-equivalent in $\mathscr{B}(G)$
(resp., in $\mathscr{B}^+(G)$) if and only if there exists $x\in G$
such that $\alpha^a_b\cdot\alpha^x_x=\alpha^c_d\cdot\alpha^x_x$.
Then Proposition~\ref{proposition-2.1}$(i)$ implies that
$\alpha^a_b\cdot\alpha^g_g=\alpha^c_d\cdot\alpha^g_g$ for all $g\in
G$ such that $g\geqslant x$ in $G$. If $g\geqslant b$ and
$g\geqslant d$ then the definition of the semigroup operation in
$\mathscr{B}(G)$ (resp., in $\mathscr{B}^+(G)$) implies that
$\alpha^a_b\cdot\alpha^g_g=\alpha^{g\cdot b^{-1}\cdot a}_g$ and
$\alpha^c_d\cdot\alpha^g_g=\alpha^{g\cdot d^{-1}\cdot c}_g$, and
since $G$ is a group we get that $b^{-1}\cdot a=d^{-1}\cdot c$.

Conversely, suppose that $\alpha^a_b$ and $\alpha^c_d$ are elements
of the semigroup $\mathscr{B}(G)$ (resp., in $\mathscr{B}^+(G)$)
such that $b^{-1}\cdot a=d^{-1}\cdot c$. Then for any element $g\in
G$ such that $g\geqslant b$ and $g\geqslant d$ in $G$ we have that
$\alpha^a_b\cdot\alpha^g_g=\alpha^{g\cdot b^{-1}\cdot a}_g$ and
$\alpha^c_d\cdot\alpha^g_g=\alpha^{g\cdot d^{-1}\cdot c}_g$, and
hence since $b^{-1}\cdot a=d^{-1}\cdot c$ we get that
$\alpha^a_b\mathfrak{C}_{mg}\alpha^c_d$. Therefore,
$\alpha^a_b\mathfrak{C}_{mg}\alpha^c_d$ in $\mathscr{B}(G)$ (resp.,
in $\mathscr{B}^+(G)$) if and only if $b^{-1}\cdot a=d^{-1}\cdot c$.

We determine a map $\mathfrak{f}\colon\mathscr{B}(G)\rightarrow G$
(resp., $\mathfrak{f}\colon\mathscr{B}^+(G)\rightarrow G$) by the
formula $(\alpha^a_b)\mathfrak{f}=b^{-1}\cdot a$, for $a,b\in G$.
Then we have that

\begin{equation*}
\begin{split}
  (\alpha^a_b\cdot\alpha^c_d)\mathfrak{f}= &
  \left\{
  \begin{array}{ll}
    (\alpha^{c\cdot b^{-1}\cdot a}_d)\mathfrak{f}, & \hbox{if } b<c; \\
    (\alpha^a_d)\mathfrak{f}, & \hbox{if } b=c; \\
    (\alpha^a_{b\cdot c^{-1}\cdot d})\mathfrak{f}, & \hbox{if } b>c,
  \end{array}
\right.
 {=} \left\{
  \begin{array}{ll}
    d^{-1}\cdot c\cdot b^{-1}\cdot a, & \hbox{if } b<c; \\
    d^{-1}\cdot a, & \hbox{if } b=c; \\
    (b\cdot c^{-1}\cdot d)^{-1}\cdot a, & \hbox{if } b>c,
  \end{array}
\right.{=} \\
= &\left\{
  \begin{array}{ll}
    d^{-1}\cdot c\cdot b^{-1}\cdot a, & \hbox{if } b<c; \\
    d^{-1}\cdot c\cdot b^{-1}\cdot a, & \hbox{if } b=c; \\
    d^{-1}\cdot c\cdot b^{-1}\cdot a, & \hbox{if } b>c,
  \end{array}
\right.
  =
     d^{-1}\cdot c\cdot b^{-1}\cdot a=
 (\alpha^c_d)\mathfrak{f}\cdot(\alpha^a_b)\mathfrak{f},
\end{split}
\end{equation*}
for $a,b,c,d\in G$. This completes the proof of the theorem.
\end{proof}

H\"{o}lder's Theorem and Theorem~\ref{theorem-2.13} imply the
following:

\begin{theorem}\label{theorem-2.14}
Let $G$ be an archimedean linearly ordered group and
$\mathfrak{C}_{mg}$ be the least group congruence on the semigroup
$\mathscr{B}(G)$ $(\mbox{resp., } \mathscr{B}^+(G))$. Then the
quotient semigroup $\mathscr{B}(G)/{\mathfrak{C}_{mg}}$
$(\mbox{resp., } \mathscr{B}^+(G)/{\mathfrak{C}_{mg}})$ is
isomorphic to the group $G$.
\end{theorem}

Theorems~\ref{theorem-2.5}, \ref{theorem-2.6} and \ref{theorem-2.14}
imply the following:

\begin{corollary}\label{corollary-2.15}
Let $G$ be an archimedean linearly ordered group and
$\mathfrak{C}_{mg}$ be the least group congruence on the semigroup
$\mathscr{B}(G)$ $(\mbox{resp., } \mathscr{B}^+(G))$. Then every
non-isomorphic image of the semigroup $\mathscr{B}(G)$
$(\mbox{resp., } \mathscr{B}^+(G))$ is isomorphic to some
homomorphic image of the group $G$.
\end{corollary}

\begin{theorem}\label{theorem-2.16}
Let $G$ be a linearly ordered $d$-group and $\mathfrak{C}_{mg}$ be
the least group congruence on the semigroup
$\overline{\mathscr{B}}(G)$ $(\mbox{resp., }
\overline{\mathscr{B}}\,^+(G))$. Then the quotient semigroup
$\overline{\mathscr{B}}(G)/{\mathfrak{C}_{mg}}$ $(\mbox{resp., }
\overline{\mathscr{B}}\,^+(G)/{\mathfrak{C}_{mg}})$ is
antiisomorphic to the group $G$.
\end{theorem}

\begin{proof}
Similar arguments as in the proofs of Theorem~\ref{theorem-2.13} and
Proposition~\ref{proposition-2.10} show that the following
assertions are equivalent:
\begin{itemize}
  \item[$(i)$] $\alpha^a_b\mathfrak{C}_{mg}\alpha^c_d$ in
   $\overline{\mathscr{B}}(G)$ (resp., in
   $\overline{\mathscr{B}}\,^+(G)$);

  \item[$(ii)$] $\alpha^a_b\mathfrak{C}_{mg}
   \overset{_\circ}{\alpha}{}^c_d$ in
   $\overline{\mathscr{B}}(G)$ (resp., in
   $\overline{\mathscr{B}}\,^+(G)$);

  \item[$(iii)$] $\overset{_\circ}{\alpha}{}^a_b\mathfrak{C}_{mg}
   \overset{_\circ}{\alpha}{}^c_d$ in
   $\overline{\mathscr{B}}(G)$ (resp., in
   $\overline{\mathscr{B}}\,^+(G)$);

  \item[$(iv)$] $b^{-1}\cdot a=d^{-1}\cdot c$.
\end{itemize}

We determine a map $\mathfrak{f}\colon\mathscr{B}(G)\rightarrow G$
(resp., $\mathfrak{f}\colon\mathscr{B}^+(G)\rightarrow G$) by the
formulae $(\alpha^a_b)\mathfrak{f}=b^{-1}\cdot a$ and
$(\overset{_\circ}{\alpha}{}^a_b)\mathfrak{f}=b^{-1}\cdot a$, for
$a,b\in G$. Then we have that
\begin{equation*}
    (\alpha^a_b\cdot\alpha^c_d)\mathfrak{f}=
    (\alpha^c_d)\mathfrak{f}\cdot(\alpha^a_b)\mathfrak{f},
\end{equation*}
\begin{equation*}
\begin{split}
  (\overset{_\circ}{\alpha}{}^a_b\cdot
  \overset{_\circ}{\alpha}{}^c_d)\mathfrak{f}= &
  \left\{
  \begin{array}{ll}
    (\overset{_\circ}{\alpha}{}^{c\cdot b^{-1}\cdot a}_d)\mathfrak{f},
          & \hbox{if } b<c; \\
    (\overset{_\circ}{\alpha}{}^a_d)\mathfrak{f},
          & \hbox{if } b=c; \\
    (\overset{_\circ}{\alpha}{}^a_{b\cdot c^{-1}\cdot d})\mathfrak{f},
          & \hbox{if } b>c,
  \end{array}
\right.
 {=} \left\{
  \begin{array}{ll}
    d^{-1}\cdot c\cdot b^{-1}\cdot a, & \hbox{if } b<c; \\
    d^{-1}\cdot a, & \hbox{if } b=c; \\
    (b\cdot c^{-1}\cdot d)^{-1}\cdot a, & \hbox{if } b>c,
  \end{array}
\right.{=} \\
    = &\ \left\{
  \begin{array}{ll}
    d^{-1}\cdot c\cdot b^{-1}\cdot a, & \hbox{if } b<c; \\
    d^{-1}\cdot c\cdot b^{-1}\cdot a, & \hbox{if } b=c; \\
    d^{-1}\cdot c\cdot b^{-1}\cdot a, & \hbox{if } b>c,
  \end{array}
\right.
 = d^{-1}\cdot c\cdot b^{-1}\cdot a=
 (\overset{_\circ}{\alpha}{}^c_d)\mathfrak{f}\cdot
 (\overset{_\circ}{\alpha}{}^a_b)\mathfrak{f},
\end{split}
\end{equation*}
\begin{equation*}
\begin{split}
  (\alpha^a_b\cdot
  \overset{_\circ}{\alpha}{}^c_d)\mathfrak{f}= &
  \left\{
  \begin{array}{ll}
    (\overset{_\circ}{\alpha}{}^{c\cdot b^{-1}\cdot a}_d)\mathfrak{f},
          & \hbox{if } b<c; \\
    (\overset{_\circ}{\alpha}{}^a_d)\mathfrak{f},
          & \hbox{if } b=c; \\
    (\alpha^a_{b\cdot c^{-1}\cdot d})\mathfrak{f},
          & \hbox{if } b>c,
  \end{array}
\right.
 {=} \left\{
  \begin{array}{ll}
    d^{-1}\cdot c\cdot b^{-1}\cdot a, & \hbox{if } b<c; \\
    d^{-1}\cdot a, & \hbox{if } b=c; \\
    (b\cdot c^{-1}\cdot d)^{-1}\cdot a, & \hbox{if } b>c,
  \end{array}
\right.{=}\\
    = &\,  \left\{
  \begin{array}{ll}
    d^{-1}\cdot c\cdot b^{-1}\cdot a, & \hbox{if } b<c; \\
    d^{-1}\cdot c\cdot b^{-1}\cdot a, & \hbox{if } b=c; \\
    d^{-1}\cdot c\cdot b^{-1}\cdot a, & \hbox{if } b>c,
  \end{array}
\right.
 = d^{-1}\cdot c\cdot b^{-1}\cdot a=
 (\alpha^c_d)\mathfrak{f}\cdot
 (\overset{_\circ}{\alpha}{}^a_b)\mathfrak{f},
\end{split}
\end{equation*}
\begin{equation*}
\begin{split}
  (\overset{_\circ}{\alpha}{}^a_b\cdot
  \alpha^c_d)\mathfrak{f}= &
  \left\{
  \begin{array}{ll}
    (\alpha^{c\cdot b^{-1}\cdot a}_d)\mathfrak{f},
          & \hbox{if } b<c; \\
    (\overset{_\circ}{\alpha}{}^a_d)\mathfrak{f},
          & \hbox{if } b=c; \\
    (\overset{_\circ}{\alpha}{}^a_{b\cdot c^{-1}\cdot d})\mathfrak{f},
          & \hbox{if } b>c,
  \end{array}
\right.
 {=} \left\{
  \begin{array}{ll}
    d^{-1}\cdot c\cdot b^{-1}\cdot a, & \hbox{if } b<c; \\
    d^{-1}\cdot a, & \hbox{if } b=c; \\
    (b\cdot c^{-1}\cdot d)^{-1}\cdot a, & \hbox{if } b>c,
  \end{array}
\right.{=}\\
    = & \,\left\{
  \begin{array}{ll}
    d^{-1}\cdot c\cdot b^{-1}\cdot a, & \hbox{if } b<c; \\
    d^{-1}\cdot c\cdot b^{-1}\cdot a, & \hbox{if } b=c; \\
    d^{-1}\cdot c\cdot b^{-1}\cdot a, & \hbox{if } b>c,
  \end{array}
\right.
 = d^{-1}\cdot c\cdot b^{-1}\cdot a=
 (\overset{_\circ}{\alpha}{}^c_d)\mathfrak{f}\cdot
 (\alpha^a_b)\mathfrak{f},
\end{split}
\end{equation*}
for $a,b,c,d\in G$.  This completes the proof of the theorem.
\end{proof}

H\"{o}lder's Theorem and Theorem~\ref{theorem-2.16} imply the
following:

\begin{theorem}\label{theorem-2.17}
Let $G$ be an archimedean linearly ordered $d$-group and
$\mathfrak{C}_{mg}$ be the least group congruence on the semigroup
$\overline{\mathscr{B}}(G)$ $(\mbox{resp., }
\overline{\mathscr{B}}\,^+(G))$. Then the quotient semigroup
$\overline{\mathscr{B}}(G)/{\mathfrak{C}_{mg}}$ $(\mbox{resp., }
\overline{\mathscr{B}}\,^+(G))/{\mathfrak{C}_{mg}})$ is isomorphic
to the group $G$.
\end{theorem}

Theorems~\ref{theorem-2.11} and \ref{theorem-2.17} imply the
following:

\begin{corollary}\label{corollary-2.18}
Let $G$ be an archimedean linearly ordered $d$-group, $T$ be a
semigroup and $h\colon \overline{\mathscr{B}}(G)\rightarrow T$
$\left(\mbox{resp., } h\colon\overline{\mathscr{B}}\,^+(G)
\rightarrow T\right)$ be a homomorphism. Then only one of the
following conditions holds:
\begin{itemize}
  \item[$(i)$] $h$ is a monomorphism;

  \item[$(ii)$] the image $\left(\overline{\mathscr{B}}(G)\right)h$
   $\left(\mbox{resp., }
   \left(\overline{\mathscr{B}}\,^+(G)\right)h\right)$ is isomorphic
   to some homomorphic image of the group $G$;

  \item[$(iii)$] the image $\left(\overline{\mathscr{B}}(G)\right)h$
   $\left(\mbox{resp., }
   \left(\overline{\mathscr{B}}\,^+(G)\right)h\right)$ is isomorphic
   to the semigroup
$\mathscr{B}(G)$ $\left(\mbox{resp., }\mathscr{B}\,^+(G)\right)$.
\end{itemize}
\end{corollary}


\section*{Acknowledgements}

This research was supported by Estonian Science Foundation an
co-funded by Marie Curie Action, grant ERMOS36 (GMTMM0036J).

The authors are grateful to the referee for several comments and
suggestions which have considerably improved the original version of
the manuscript.


\end{document}